\documentclass[article,11pt]{article}
\usepackage{graphicx}
\usepackage{amsmath}
\usepackage{mathrsfs,amsfonts}
\usepackage{amsfonts}
\usepackage{amsmath,amssymb,amsfonts}
\usepackage{amssymb}
\usepackage{color}
\usepackage{amsmath}  
\usepackage{amsfonts} 
\usepackage{graphicx} 
\usepackage{braket}
\usepackage{enumerate}
\usepackage{indentfirst}
\usepackage{longtable}
\usepackage{setspace}
\allowdisplaybreaks[4]

\usepackage{amsmath,amsfonts,amsthm,amssymb}    
\usepackage{gensymb, mathrsfs}                  

\usepackage{accents}
\DeclareMathSymbol{\widehatsym}{\mathord}{largesymbols}{"62}

\def\*#1{\mathbf{#1}}

\renewcommand{\bar}{\overline}
\renewcommand{\tilde}{\widetilde}
\renewcommand{\phi}{\varphi}

\makeatletter
\newcommand{\figcaption}{\def\@captype{figure}\caption}
\newcommand{\tabcaption}{\def\@captype{table}\caption}
\makeatother

\newfam\msbmfam\font\tenmsbm=msbm10\textfont
\msbmfam=\tenmsbm\font\sevenmsbm=msbm7
\scriptfont\msbmfam=\sevenmsbm\def\bb#1{{\fam\msbmfam #1}}

\def\EE{\mathbb E}
\def\RR{\mathbb R}

\def\<{\left<}\def\>{\right>}

\def\({\left(}\def\){\right)}
\def\[{\left[}\def\]{\right]}

\makeatletter \renewcommand\d[1]{\ensuremath{%
  \;\mathrm{d}#1\@ifnextchar\d{\!}{}}}
\makeatother

\theoremstyle{plain}

\theoremstyle{definition}

\newtheorem{example}{Example}[section]
\usepackage{xpatch,amsthm}
\makeatletter
   \xpatchcmd{\@thm}{\fontseries\mddefault\upshape}{}{}{} 
\makeatother

\newcommand{\beq}{\begin{equation}}
\newcommand{\eeq}{\end{equation}}

\definecolor{c}{rgb}{0.9,0.3,0.1}
\definecolor{b}{rgb}{0.1,0.3,0.9}

\newtheorem{remark}{Remark}[section]
\newtheorem{lemma}{Lemma}[section]
\newtheorem{theorem}{Theorem}[section]
\newtheorem{definition}{Definition}[section]
\newtheorem{hypothesis}{Hypothesis}[section]

\def\<{\left<}\def\>{\right>}\def\({\left(}\def\){\right)}

\newfam\msbmfam\font\tenmsbm=msbm10\textfont
\msbmfam=\tenmsbm\font\sevenmsbm=msbm7
\scriptfont\msbmfam=\sevenmsbm\def\bb#1{{\fam\msbmfam #1}}

\def\EE{\bb E}
\def\RR{\bb R}

\numberwithin{equation}{section}

\begin{document}

\title{\large \bf Stochastic maximum principle for generalized mean-field delay control problem \thanks{Research supported partially by FDCT 025/2016/A1.}}
\author{Hancheng Guo\footnotemark[2], \ \ Jie Xiong\footnotemark[2] \ \ and \ \ Jiayu Zheng\thanks{Department of Mathematics, Faculty of Science and Technology, University of Macau, Macau, China. E-mail: guohancheng1989@gmail.com,  jiexiong@umac.mo and jy\_zheng@outlook.com.}
}
\date{}
\maketitle
 \bigskip
 
\noindent \textbf{Abstract.}
 In this paper, we first give the existence and uniqueness theorems for generalized mean-filed delay stochastic differential equations (GMFDSDEs) and mean-field anticipated backward stochastic differential equations (MFABSDEs). Then we study the stochastic maximum principle for generalized mean-filed delay control problem. Since the state is distribution-depending, we define the adjoint equation as a MFABSDE, in which, all the derivatives of coefficients are in Fr\'echet sense. We deduce the stochastic maximum principle, and also obtain, under some additional assumptions, a sufficient condition for the optimality of the control.

\bigskip


\section{{\protect  {Introduction}}}
In this paper we discuss new types of differential equatios which we call mean-field anticipated backward stochastic differential equations (MFABSDEs):
\begin{eqnarray}\label{eqc2}
\begin{cases}
-dY_t = \EE^{'}[f(t, Y'_t, Z'_t, Y'_{t+\delta(t)}, Z'_{t+\zeta(t)}, Y_t, Z_t, Y_{t+\delta(t)}, Z_{t+\zeta(t)})]dt - Z_t dB_t  , \\
\ \ \ \ \ \ \ \ \ \ \ \ \ \ \ \ \ \ \ \ \ \ \ \ \ \ \ \ \ \ \ \ \ \ \ \ \ \ \ \ \ \ \ \ \ \ \ \ \ \ \ \ \ \ \ \ \ \ \ \ \ \ \ \ \ \ \ \ \ \ \ \ \ \ \ \ \ \ \ \   t \in [0,T];
\cr 
Y_t = \xi_t,   Z_t = \eta_t , \ \ \ \ \ \ \ \ \ \ \ \ \ \ \ \ \  \ \ \ \ \ \ \ \ \ \ \ \ \ \ \ \ \ \ \ \ \ \ \ \ \ \ \ \ \ \ \ \ \ \ \      t \in  [T, T+K],
\end{cases}
\end{eqnarray}
where $B$ is a $d$-dimensional Brownian motion, $K$ is a constant. Precise assumptions on the coefficient $f$ and the definition of $\EE'$ are given in the following sections.

Actually, the above MFABSDE is inspired by the mean-field BSDEs 
\begin{align} \label{eqq3}
Y_t = \xi_T + \int_t^T \EE^{'} \[f(s, Y_s, Z_s, Y_s^{'}, Z_s^{'})\]ds - \int_t^T Z_s dB_s, \ \ \ \ \  t \in [0, T]
\end{align}
that is studied by \cite{BLP2009} and the anticipated BSDEs
\begin{eqnarray} \label{eqq4}
\begin{cases}
Y_t = \xi + \int_t^T f(s, Y_s, Z_s, Y_{s + \delta(s)}, Z_{s + \zeta(s)})ds - \int_t^T Z_s dB_s,   \ \ \ \  t \in [0,T];
\cr 
Y_t = \xi_t,    \ \ \ \ \ \ \ \ \ \ \ \ \ \ \ \ \ \ \ \ \ \ \ \ \ \ \ \ \ \ \ \ \ \ \ \ \ \ \ \ \ \  \ \ \ \ \ \ \ \ \ \ \ \ \ \ \ \   t \in  (T, T+K];
\cr
Z_t = \eta_t,    \ \ \ \ \ \ \ \ \ \ \ \ \ \ \ \ \ \ \ \ \ \ \ \ \ \ \ \ \ \ \ \ \ \ \ \ \ \ \ \ \ \  \ \ \ \ \ \ \ \ \ \ \ \ \ \ \ \   t \in  (T, T+K]. 
\end{cases}
\end{eqnarray}
which is investigated by \cite{PY}.

\indent We consider the stochastic maximum principle for a generalized mean-field delay control problem, whose state equation is defined as
\begin{eqnarray}\label{2eq11}
\left\{
\begin{array}{l}
dX^v_t=b(t,X^v_t,X^v_{t-\delta}, P_ {{X^v_t}},P_ {{X^v_{t-\delta}}},v_t,v_{t-\delta}) dt+\sigma  dB_t, \ \ \ t\in[0,T],\\
X^v_t=\xi_t, \ v_t=\eta_t, \ \ \ \ \ \ \ \ \ \ \ \ \ \ \ \ \ \ \ \ \ \ \ \ \ \ \ \ \ \ \ \ \ \ \ \ \ \ \ \ \ \ \ \ \ \ \ t\in[-\delta,0],
\end{array}
\right.
\end{eqnarray}
where $P_ {{X}}$ is the law of ${X}$, $\delta\in [0,T]$, $\sigma$ has the same structure as $b$, and the cost functional is defined as:
\begin{equation}\label{2eq12}
  \begin{aligned}
J(v)=\mathbb{E}\left\{ \int^T_0 h(t,X^v_t,P_ {{X^v_t}},v_t,v_{t-\delta})dt+\Phi(X^v_T,P_ {{X^v_T}})\right\}.
\end{aligned}
 \end{equation}
The agent wishes to minimize his cost functional $J(v)$. Namely, an admissible control $u\in \mathcal{U}$ is said to be optimal if
$$J(u)=\min_{v\in \mathcal{U}}J(v).$$ 
\indent About stochastic maximum principle (SMP), some pioneering works have been done by  Pontryagin et al. \cite{P1962}, they obtained the Pontryagin's maximum principle by using ``spike variation''.  Kushner \cite{K1965} \cite{K1972} studied the SMP in the framework when the diffusion coefficient does not depend on the control variable, and the cost functional consists of terminal cost only. Haussmann \cite{H1986} gave a version of SMP when the diffusion of the state does not depend on control item. Arkin and Saksonov \cite{AS1979}, Bensoussan  \cite{B1982} and Bismut \cite{B1978}, proved different versions of SMP under various setups.\\
\indent Pardoux and Peng \cite{PP1990} introduced non-linear backward stochastic differential equations (BSDE) in 1990. They showed that under appropriate assumptions, BSDE admits an unique adapted solution, and the associated comparison theorem holds. An SMP was obtained by Peng \cite{P1990} in the same year. In that paper, first and second order variational inequalities are introduced, when the control domain need not to be convex, and the diffusion coefficient contains the control variable. The authors of \cite{BLP2009} obtained mean-field BSDE in a natural way as the limit of some high dimensional system of forward and backward stochastic differential equations. Li \cite{L2012} studied SMP for mean-filed controls, when the domain of the control is assumed to be convex. Under some additional assumptions, both necessary and sufficient conditions for the optimality of a control were proved.\\
\indent Buckdahn et al. \cite{BDL2011} considered an SMP for SDEs of mean-field type control problem when the coefficients depend on the state of the solution process as well as on its expected value.
 Moreover, the cost functional is also of mean-field type. Their system is defined as follows:
 \begin{equation}
\left\{
\begin{array}{l}
dX_t=b(t,X_t,\mathbb{E}[{{X_t}}],v_t) dt
    +\sigma(t,X_t,\mathbb{E}[{{X_t}}],v_t) dB_t,\\
X_0=x.
\end{array}
\right.
\end{equation}
and the cost/payoff functional is defined by:
\begin{equation}
\begin{split}
  \begin{aligned}
J(v)=\mathbb{E}\left\{ \int^T_0 h(t,X_t,\mathbb{E}[{{X_t}}],v_t)dt+\Phi(X_T,\mathbb{E}[{{X_T}}])\right\},
\end{aligned}
\end{split}
 \end{equation}
  An SMP is derived, specifying the necessary conditions for the optimality. This maximum principle differs from the classical one in the sense that here the first order adjoint equation turns out to be a linear mean-field backward SDE, while the second order adjoint equation remains the same as in Peng's SMP. About stochastic delay control problem, Chen and Wu \cite{CW2010} obtain the maximum principle for the optimal control of this problem by virtue of the duality method and the anticipated backward stochastic differential equations. The Authors of \cite{DHQ2013} develop this theory into classical mean-field type, which means the coefficients of the state depend on the expectation.\\
 \indent Buckdahn et al. \cite{BLP2014} studied generalized mean-field stochastic differential equations and the associated partial differential equations (PDEs). ``Generalized" means the coefficients depend on both the state process and its law. They proved that under appropriate regularity conditions on the coefficients, the SDE has the unique classical solution. In this paper, we study the optimal control when the state equation is in the controlled generalized mean-filed form.\\

\section{Preliminaries}

\indent In this section, for the convenience of the reader, we state some results of Buckdahn et al. \cite{BLP2014} without proofs, which will be used in present work.

Let $(\Omega,{\cal{F}},P)$ be a probability space with filtration ${\cal{F}}_{t}$.
Suppose that ${B}_t$
 is a Brownian motion belongs to
$(\Omega,{\cal{F}},P)$, where ${\mathcal{F}}$ is the filtration generated by ${B}_t$, and augmented by all $P$-null sets. Let $\mathcal{P}_2(\mathbb{R}^n)$  be the collection of all square integrable probability measures over $(\mathbb{R}^n,\mathcal{B}(\mathbb{R}^n))$, endowed with the 2-Wasserstein metric $W_2$, which is defined as
$$W_2(P_{\mu },P_{\nu})=\inf\left\{ \left(\mathbb{E}[|\mu'-\nu'|^2]\right)^{\frac{1}{2}}\right \},$$
for all $\mu',\nu' \in L^2(\mathcal{F}_0;\mathbb{R}^d)$ with $P_{\mu'}=P_{\mu}, \ P_{\nu'}=P_{\nu}.$ Now let us introduce the following spaces: 
\begin{eqnarray*}
L^p(\Omega, \mathcal{F}_T, P;\mathbb{R}^n)=\left\{\xi:\mathbb{R}^n\text{-valued} \ \mathcal{F}_T\text{-measurable\  r.v.} ; \mathbb{E}\left[|\xi|^p\right]<+\infty\right\},
\end{eqnarray*}
\begin{eqnarray*}
L^0(\Omega, \mathcal{F}, P;\mathbb{R}^n)=\left\{\xi:\mathbb{R}^n\text{-valued} \ \mathcal{F}\text{-measurable\  random variables} \right\},
\end{eqnarray*}
\begin{eqnarray*}
H^2_{\mathcal{F}}(s, r;\mathbb{R}^n)=&\bigg\{&(\phi_t)_{s\leq t\leq r}: \mathbb{R}^n\text{-valued} \ \mathcal{F}_t\text{-adapted stochastic process} ;\\
 &&\mathbb{E}\left[\int^r_s |\phi_t|^2dt\right]<+\infty\bigg\},
\end{eqnarray*}
\begin{eqnarray*}
S^2_{\mathcal{F}}(s, r;\mathbb{R}^n)=&\bigg\{&(\phi_t)_{s\leq t\leq r}: \mathbb{R}^n\text{-valued} \ \mathcal{F}_t\text{-adapted stochastic process} ;\\
 &&\mathbb{E}\left[\sup_{s\leq t\leq r} |\phi_t|^2\right]<+\infty\bigg\}.
\end{eqnarray*}
 $\mathcal{U}=H^2_{\mathcal{F}}(0,T;U)$ denotes the set of admissible controls of the following form:
 \begin{eqnarray*}
v_t=\left\{
\begin{array}{l}
v_t\in H^2_{\mathcal{F}}(0, T;\mathbb{R}^n), \ t\in[0,T],\\
\gamma_t,\ \ \ \ \ \ \ \ \ \ \ \ \  \ \ \ \ \ \ \  \ t\in[-\delta,0],
\end{array}
\right.
\end{eqnarray*}
  where, $\gamma$ is square integrable on $[-\delta,0]$, $U$ is supposed to be a convex subset of $\mathbb{R}^k$. Given $b:[0,T]\times \mathbb{R}^n\times \mathbb{R}^n\times \ \mathcal{P}_2(\mathbb{R}^n)\times \ \mathcal{P}_2(\mathbb{R}^n) \times U\times U \longrightarrow \mathbb{R}^n,\ \sigma:[0,T]\times \mathbb{R}^n\times \mathbb{R}^n\times \ \mathcal{P}_2(\mathbb{R}^n)\times \ \mathcal{P}_2(\mathbb{R}^n) \times U\times U \longrightarrow \mathbb{R}^{n\times d}, \ h:[0,T]\times \mathbb{R}^n\times \ \mathcal{P}_2(\mathbb{R}^n) \times U \times U\longrightarrow \mathbb{R}, \  \Phi:\mathbb{R}^n\times \mathcal{P}_2(\mathbb{R}^n) \longrightarrow \mathbb{R}.$ 

 About the deriavative with respect to measure, the following definition is taken from Cardaliaguet \cite{C2013}.
\begin{definition}
A function $f: \mathcal{P}_2(\mathbb{R}^n)\longrightarrow \mathbb{R} $ is said to be differentiable in $\mu \in \mathcal{P}_2(\mathbb{R}^n)$ if, the function $\tilde{f}: L^2(\mathcal{F};\mathbb{R}^n)\longrightarrow \mathbb{R}$ given by $\tilde{f}(\mathfrak{v})=f(P_{\mathfrak{v}})$ is differentiable (in Fr\'{e}chet sense) at $\mathfrak{v}_0$, defined by $P_{\mathfrak{v}_0}=\mu$, i.e. there exists a linear continuous mapping $D\title{f}(\mathfrak{v}_0):L^2( \mathcal{F} ;\mathbb{R}^n)\longrightarrow \mathbb{R},$ such that
$$\tilde{f}(\mathfrak{v}_0+\eta)-\tilde{f}(\mathfrak{v}_0)=D\tilde{f}(\mathfrak{v}_0)(\eta)+o(|\eta|_{L^2}),$$
with $|\eta|_{L^2}\longrightarrow 0$ for $\eta \in L^2( \mathcal{F} ;\mathbb{R}^n). $
\end{definition}

\indent According to Riesz' Representation Theorem, there exists a unique random variable $\theta_0\in L^2( \mathcal{F} ;\mathbb{R}^n)$ such that $D\tilde{f}(\mathfrak{v}_0)(\eta)=(\theta_0,\eta)_{L^2}=\mathbb{E}[\theta_0\eta]$, for all $\eta\in  L^2( \mathcal{F} ;\mathbb{R}^n).$ In \cite{C2013} it has been proved that there is a Borel function $h_0:\mathbb{R}^d\longrightarrow\mathbb{R}^d$ such that $\theta_0=h_0(\mathfrak{v}_0),\ \  \ a.s..$ Then,
$$f(P_{\mathfrak{v}})-f(P_{\mathfrak{v}_0})=\mathbb{E}[h_0(\mathfrak{v}_0)(\mathfrak{v}-\mathfrak{v}_0)]+o(|\mathfrak{v}-\mathfrak{v}_0|_{L^2}),$$
$\mathfrak{v}\in L^2( \mathcal{F} ;\mathbb{R}^n).$\\
\indent We call $\partial_{\mu} f(\mu ,y):=h_0(y), \ y\in \mathbb{R}^n $, the derivative of $f: \mathcal{P}_2(\mathbb{R}^n)\longrightarrow \mathbb{R}^n$ at $\mu.$

\indent For mean-field type SDE and BSDE, we introduce the following notations. Let $(\Omega',{\cal{F}}',P')$ be a copy of the probability space $(\Omega,{\cal{F}},P)$. For each random variable $\xi$ over $(\Omega,{\cal{F}},P)$ we denote by $\xi'$ a copy of $\xi$ defined over $(\Omega',{\cal{F}}',P')$. $\mathbb{E}'[\cdot]=\int_{\Omega'}(\cdot)dP'$ acts only over the variables from $(\Omega',{\cal{F}}',P')$.

\indent Recall that for 2-Wasserstein metric $W_2(\cdot,\cdot),$ we have,
$$W_2(P_{\mu },P_{\nu})=\inf\{ (\mathbb{E}[|\mu'-\nu'|^2])^{\frac{1}{2}} \},$$
for all $\mu',\nu' \in L^2(\mathcal{F}_0;\mathbb{R}^d)$ with $P_{\mu'}=P_{\mu}, \ P_{\nu'}=P_{\nu}.$
\begin{definition}
We say that $f\in C^{1,1}_b(\mathcal{P}_2(\mathbb{R}^d))$  (continuously differentiable over $\mathcal{P}_2(\mathbb{R}^d)$ with Lipschitz-continuous bounded derivative), if for all $\mathfrak{v}\in L^2(\mathcal{F},\mathbb{R}^d)$, there exists a $P_{\mathfrak{v}}$-modification of $\partial_{\mu }f(P_{\mathfrak{v}},\cdot)$, again denote by $\partial_{\mu }f(P_{\mathfrak{v}},\cdot)$, such that $\partial_{\mu }f:\mathcal{P}_2(\mathbb{R}^d)\times \mathbb{R}^d\longrightarrow \mathbb{R}^d$ is bounded and Lipschitz continuous, i.e., there is some real constant $C$ such that
\begin{equation}
  \begin{aligned}
i)\ &|\partial_{\mu }f(\mu ,x)|\leq C, \mu \in\mathcal{P}_2( \mathbb{R}^d), x\in \mathbb{R}^d,&\\
ii)\ &|\partial_{\mu }f(\mu ,x)-\partial_{\mu }f(\mu' ,x')|\leq C(W_2(\mu,\mu')+|x-x'|), \mu, \mu' \in\mathcal{P}_2( \mathbb{R}^d),&
 \end{aligned}
\end{equation}
$ x,x'\in \mathbb{R}^d.$ We consider this function $\partial_{\mu }f$ as the derivative of $f$. 
\end{definition}

\indent Let us now consider a complete probability space $(\Omega,\mathcal{F},P)$ on which, we define a $d$-dimensional Brownian motion $B=(B^1,\cdots,B^d)=(B_t)_{t\in [0,T]}$, and $T\geq 0$ denotes an arbitrarily fixed time horizon. We make the following assumptions:\\
 \indent There is a sub-$\sigma$-field $\mathcal{F}_0\subset\mathcal{F}$ such that\\
\indent i) the Brownian motion $B$ is independent of $\mathcal{F}_0$, and\\
\indent ii) $\mathcal{F}_0$ is ``rich enough", i.e., $\mathcal{P}_2(\mathbb{R}^d)=\{ P_{\mathfrak{v}}, \mathfrak{v}\in L^2(\mathcal{F}_0; \mathbb{R}^d) \}.$\\
By $\mathbb{F}=(\mathcal{F}_t)_{t\in [0,T]}$ we denote the filtration generated by $B$, completed and augmented by $\mathcal{F}_0$.\\
\indent Given deterministic Lipschitz functions $\sigma: \mathbb{R}^d\times\mathcal{P}_2(\mathbb{R}^d)\longrightarrow \mathbb{R}^{d\times d}$ and  $b: \mathbb{R}^d\times\mathcal{P}_2(\mathbb{R}^d)\longrightarrow \mathbb{R}^{d},$ we consider for the initial state $(t,x)\in [0,T]\times \mathbb{R}^d$ and $\xi\in L^2(\mathcal{F}_t;\mathbb{R}^d)$ the stochastic differential equations (SDEs)

\begin{equation}
  \begin{aligned}
X^{t,\xi}_s=\xi+\int^s_t \sigma(X^{t,\xi}_r,P_{X^{t,\xi}_r})dB_r+\int^s_t \sigma(X^{t,\xi}_r,P_{X^{t,\xi}_r})dr,\ s\in[t,T],
  \end{aligned}
\end{equation}
and
\begin{equation}
  \begin{aligned}
  X^{t,x,\xi}_s=x+\int^s_t \sigma(X^{t,x,\xi}_r,P_{X^{t,\xi}_r})dB_r+\int^s_t \sigma(X^{t,x,\xi}_r,P_{X^{t,\xi}_r})dr,\ s\in[t,T].
  \end{aligned}
\end{equation}

We find out that under the assumptions above, both SDEs have a unique solution in $\mathcal{S}^2([t,T];\mathbb{R}^d),$ which is the space of $\mathbb{F}$-adapted continuous processes $Y=(Y_s)_{s\in[t,T]}$ with $\mathbb{E}[\sup_{s\in [t,T]}|Y_s|^2]\leq \infty.$

\begin{hypothesis}
The couple of coefficient $(\sigma,b)$ belongs to $C^{1,1}_b(\mathbb{R}^d\times \mathcal{P}_2(\mathbb{R}^d)\longrightarrow \mathbb{R}^{d\times d}\times \mathbb{R}^d),$ i.e., the components $\sigma_{i,j},b_j,\ 1\leq i,j \leq d,$ satisfy the following conditions:\\
\indent i) $\sigma_{i,j}(x,\cdot), b_{j}(x,\cdot)$ belong to $C^{1,1}_b( \mathcal{P}_2(\mathbb{R}^d))$, for all $x\in \mathbb{R}^d$\\
\indent ii) $\sigma_{i,j}(\cdot,\mu), b_{j}(\cdot,\mu)$ belong to $C^{1}_b(\mathbb{R}^d)$, for all $\mu\in \mathcal{P}_2(\mathbb{R}^d)$\\
iii) The derivatives $\partial_x\sigma_{i,j}, \partial_xb_{j}:\mathbb{R}^d\times \mathcal{P}_2(\mathbb{R}^d)\longrightarrow \mathbb{R}^d $, $\partial_{\mu}\sigma_{i,j}, \partial_{\mu}b_{j}:\mathbb{R}^d\times \mathcal{P}_2(\mathbb{R}^d)\times \mathbb{R}^d\longrightarrow \mathbb{R}^d $,  are bounded and Lipschitz continuous.

\end{hypothesis}

\begin{hypothesis}\label{H22}
The couple of coefficient $(\sigma,b)$ belongs to $C^{2,1}_b(\mathbb{R}^d\times \mathcal{P}_2(\mathbb{R}^d)\longrightarrow \mathbb{R}^{d\times d}\times \mathbb{R}^d),$ i.e., $(\sigma,b)\in C^{1,1}_b(\mathbb{R}^d\times \mathcal{P}_2(\mathbb{R}^d)\longrightarrow \mathbb{R}^{d\times d}\times \mathbb{R}^d)$ and the components $\sigma_{i,j},b_j,\ 1\leq i,j \leq d,$ satisfies the following conditions:\\
\indent i) $\partial_{x_k}\sigma_{i,j}(\cdot,\cdot), \partial_{x_k} b_{j}(\cdot,\cdot)$ belong to $C^{1,1}_b( \mathbb{R}^d \times\mathcal{P}_2(\mathbb{R}^d))$, for all $1\leq k \leq d;$\\
\indent ii) $\partial_{\mu }\sigma_{i,j}(\cdot,\cdot,\cdot), \partial_{\mu }b_{j}(\cdot, \cdot,\cdot)$ belong to $C^{1,1}_b(\mathbb{R}^d \times \mathcal{P}_2(\mathbb{R}^d)\times \mathbb{R}^d)$, for all $\mu\in \mathcal{P}_2(\mathbb{R}^d)$\\
iii) All the derivatives of $\sigma_{i,j}, b_{j}$, up to order 2 are bounded and Lipschitz continuous.
\end{hypothesis}

\indent The following theorem is taken from \cite{BLP2014}. It gives the It\^{o}'s formula related to a probability measure.

\begin{theorem}
Let $\Phi \in C^{2,1}_b(\mathbb{R}^d\times \mathcal{P}_2(\mathbb{R}^d)).$ Then, under Hypothesis \ref{H22}, for all $0\leq t \leq s \leq T, x\in \mathbb{R}^d, \xi\in L^2(\mathcal{F}_t;\mathbb{R}^d)$ the It\^o formula is satisfied as follow:
\begin{eqnarray}
&&\Phi(X^{t,x,P_{\xi}}_s, P_{X^{t,\xi}_s})-\Phi(x,P_{\xi})\nonumber\\
&=&\int^s_t\big(\sum^d_{i=1}\partial_{x_i}\Phi(X^{t,x,P_{\xi}}_r, P_{X^{t,\xi}_r})b_i(X^{t,x,P_{\xi}}_r, P_{X^{t,\xi}_r})\nonumber\\
&&+\frac{1}{2} \sum^d_{i,j,k=1} \partial^2_{x_i,x_j}\Phi(X^{t,x,P_{\xi}}_r, P_{X^{t,\xi}_r})(\sigma_{i,k}\sigma_{j,k})(X^{t,x,P_{\xi}}_r, P_{X^{t,\xi}_r})\nonumber\\
&&+\mathbb{E}'\big[\sum^d_{i=1}(\partial_{\mu }\Phi)_i(X^{t,x,P_{\xi}}_r, P_{X^{t,\xi}_r},(X^{t,{\xi}}_r)')b_i((X^{t,{\xi}}_r)', P_{X^{t,\xi}_r})\nonumber\\
&&+\frac{1}{2} \sum^d_{i,j,k=1} \partial_{y_i}((\partial_{\mu }\Phi)_j(X^{t,x,P_{\xi}}_r, P_{X^{t,\xi}_r},(X^{t,{\xi}}_r)')(\sigma_{i,k}\sigma_{j,k})((X^{t,{\xi}}_r)', P_{X^{t,\xi}_r})\big]\big)dr\nonumber\\
&&+\int^s_t \sum^d_{i,j=1}\partial_{x_i}\Phi(X^{t,x,P_{\xi}}_r, P_{X^{t,\xi}_r})\sigma_{i,j}(X^{t,x,P_{\xi}}_r, P_{X^{t,\xi}_r})dB^j_r, \ s\in [t,T].
\end{eqnarray}

\end{theorem}
\indent For mean-field type SDE and BSDE, we have still to introduce some notations. Let $({\Omega}',{\mathcal{F}}',{P}')$ be a copy of the probability space $(\Omega,{\mathcal{F}},P)$. For any random variable $\xi$ over $(\Omega,{\mathcal{F}},P)$, we denote by ${\xi}'$ its copy on ${\Omega}'$, respectively, which means that they have the same law as $\xi$, but defined over  $({\Omega}',{\mathcal{F}}',{P}')$ . $\mathbb{{E}}'[\cdot]=\int_{{\Omega}'}(\cdot)d{P'}$ act only over the variables from  ${\omega}'$.
\indent About stochastic delay and anticipated differential equations, we would like to introduce the following lemmas for the convenience of the readers. Our Lemma \ref{le3} is Lemma 3.1 of Peng \cite{P2004}. Lemma \ref{le1}, which is Theorem 3.1 of Buckdahn \cite{BLP2009} , is a fundamental result of mean-filed BSDEs: an existence and uniqueness theorem. Lemma \ref{le2} is the comparison theorem for solutions of mean-filed BSDEs that can be found in Buckdahn \cite{BLP2009}.

\begin{lemma}\label{le3}
For a fixed $\xi \in L^2(\mathscr{F}_T)$ and $g_0(\cdot)$ which is an $\mathscr{F}_t$-adapted process satisfying $\EE[(\int_0^T | g_0(t) |dt)^2] <\infty$, there exists a unique pair of processes $(y., z.) \in H^2_{\mathscr{F}}(0, T; \mathbb{R}^{1+d})$ satisfying the following BSDE:
\begin{align*}
y_t = \xi + \int_t^T g_0(s) ds - \int_t^T z_s dW_s, \quad t \in [0, T].
\end{align*}
If  $g_0(\cdot) \in L_{\mathscr{F}^2(0, T)}$, then $(y., z.) \in S^2_{\mathscr{F}}(0, T) \times H^2_{\mathscr{F}}(0, T; \mathbb{R}^d).$ We have the following basic estimate:
\begin{align}\label{estimates1}
|y_t|^2 &+ \EE^{\mathscr{F}_t} \[ \int_t^T \( \frac{\beta}{2} |y_s|^2 + |z_s|^2 \) e^{\beta(s-t)} ds \] \nonumber \\
& \le  \EE^{\mathscr{F}_t} \[ |\xi|^2 e^{\beta (T- t)} \] + \frac{2}{\beta} \EE^{\mathscr{F}_t} \[ \int_t^T |g_0(s)|^2 e^{\beta(s-t)} ds \]
\end{align}
In particular,
\begin{align}\label{eatimates2}
|y_0|^2 &+ \EE \[ \int_0^T \( \frac{\beta}{2} |y_s|^2 + |z_s|^2 \) e^{\beta s} ds \]  \nonumber \\
& \le  \EE \[ |\xi|^2 e^{\beta T} \] + \frac{2}{\beta} \EE \[ \int_0^T |g_0(s)|^2 e^{\beta s} ds \]
\end{align}
where $\beta >0$ is an arbitrary constant.
\end{lemma}

The following is a foundamental result for the existence of a unique solution to mean-filed BSDEs due to Buckdahn \cite{BLP2009}.(Theorem 3.1)

\begin{lemma}\label{le1}
Under the assumptions (H1) and (H2) of \cite{PY}, and $\delta$, $\zeta$ satisfy (C1) and (C2). Then for any terminal conditions $\xi \in S^2_{\mathscr{F}}(T, T+K; \RR^m)$ and $\eta \in L^2_{\mathscr{F}}(T, T+K; \RR^{m \times d})$, the anticipated BSDE (\ref{eqq4})
has a unique adapted solution
\begin{align*}
 (Y_t, Z_t) \in \mathcal{S}^2_{\mathscr{F}}(0, T+K; \RR^m) \times H^2_{\mathscr{F}}(0, T+K;  \RR^{m\times d}). 
\end{align*}
\end{lemma}
\begin{remark}
We emphasize that, due to our notation, the driving coefficient of (\ref{eqq3}) has to be interpreted as followings:
\begin{align*}
\mathbb{E}^{'}[f(s, Y_s^{'}, Z_s^{'}, Y_s, Z_s)](\omega) & = \mathbb{E}^{'}[f(s, Y_s^{'}, Z_s^{'}, Y_s(\omega), Z_s(\omega))]  \nonumber \\
&= \int_{\Omega} f(\omega^{'}, \omega, s, Y_s^{'}(\omega^{'}), Z_s^{'}(\omega^{'}), Y_s(\omega), Z_s(\omega)) P(d \omega^{'}).
\end{align*}
\end{remark}

The proof of the following comparison theorem for mean filed BSDE can be found in Buckdahn \cite{BLP2009}.

\begin{lemma}\label{le2}
Let $\bar{f}_{i}(t, y, z, y', z'), i=1, 2,$ be two drivers of mean-filed BSDEs satisfying the the assumptions (A3) and (A4) of \cite{BLP2009}. Moreover, suppose:
\begin{enumerate}[(i)]
\item  One of the two coefficients is independent of $z'$.

\item One of the two coefficients is nondecreasing in $y'$.
\end{enumerate}
Let $\xi_1, \xi_2 \in L^2(\Omega, \mathscr{F}_T, P)$ and denote by $(Y^1, Z^1)$ and $(Y^2, Z^2)$ the solution of the mean-field BSDE (\ref{eqq3}) with data $(\xi_1, f_1)$ and $(\xi_2, f_2)$, respectively. Then of $\xi_1 \ge \xi_2,$ P-a.s., and $f_1 \le f_2,$ $\bar{P}$-a.s., it holds that also $Y_t^1 \le Y_t^2, t \in [0, T],  $ P- a.s.
\end{lemma}

\section{Basic properties of GMFDSDE and MFABSDE}

\subsection{Existence and uniqueness theorems}

\indent Consider equation (\ref{eqc2}), where $\delta(\cdot)$ and $\zeta(\cdot)$ are two $\mathbb{R}^{+}$-valued continuous functions defined on $[0,T]$ such that:

(C1) There exists a constant $K \ge 0$ such that, for all $s \in [0, T]$,
$$
s+\delta(s) \le T+K; \quad s+ \zeta(s) \le T+K,
$$
\indent (C2) There exists a constant $L \ge 0$ such that, for all $t \in [0, T]$ and for all nonnegative and integrable $g(\cdot)$, 
$$
\int_t^T g(s+\delta(s))ds \le L \int_t^{T+K} g(s) ds;
$$
$$
\int_t^T g(s+\zeta(s))ds \le L \int_t^{T+K} g(s) ds.
$$

The setting of our problem is as follows: to find a pair of $\mathscr{F}$-adapted processes $(Y_., Z_.) \in S^2_{\mathscr{F}}(0, T+K; \mathbb{R}^m) \times H^2_{\mathscr{F}}(0, T+K; \mathbb{R}^{m \times d}) $ satisfying MFABSDE (\ref{eqc2}).

Assume that for all $s\in [0,T]$, $f(s, \cdot): \mathbb{R}^m \times \mathbb{R}^{m \times d} \times H^2(\mathscr{F}_r; \mathbb{R}^m) \times H^2(\mathscr{F}_{\bar{r}}; \mathbb{R}^{m\times d} ) \times \mathbb{R}^m \times \mathbb{R}^{m \times d} \times H^2(\mathscr{F}_r; \mathbb{R}^m) \times H^2(\mathscr{F}_{\bar{r}}; \mathbb{R}^{m\times d}) \longrightarrow H^2(\mathscr{F}_s, \mathbb{R}^m)$, where $r, \bar{r} \in [s, T+K]$, and $f$ satisfies the following conditions:

(C3) There exists a constant $C >0$, such that for all $s \in [0,T], y_1, y_2, y'_1, y'_2 \in \mathbb{R}^m, z_1, z_2, z'_1, z'_2 \in \mathbb{R}^{m \times d}, \theta_{.,1}, \theta_{.,2}, \theta'_{.,1}, \theta'_{.,2} \in L^2_{\mathscr{F}}(s, T+K; \mathbb{R}^m), \gamma_{.,1}, \gamma_{.,2}, \gamma'_{.,1}, \gamma'_{.,2} \in L^2_{\mathscr{F}}(s, T+K; \mathbb{R}^{m \times d}), r, \bar{r} \in [s, T+K]$, we have  
\begin{align*}
& |f(t, y_1, z_1,\theta_{r,1}, \gamma_{\bar{r},1}, y'_1, z'_1, \theta'_{r,1}, \gamma'_{\bar{r},1})   
-  f(t, y_2, z_2,\theta_{r,2}, \gamma_{\bar{r},2}, y'_2, z'_2, \theta'_{r,2}, \gamma'_{\bar{r},2}) |     \nonumber \\
&\le C \[ |y_1 - y_2| + |z_1 - z_2| + \EE^{\mathscr{F}_s}\(|\theta_{r,1} - \theta_{r,2}| + |\gamma_{\bar{r},1} - \gamma_{\bar{r},2}|\)  \right.       \nonumber \\
&+ |y'_1 - y'_2| + |z'_1 - z'_2| + \EE^{\mathscr{F}_s}\( |\theta'_{r,1} - \theta'_{r,2}| + |\gamma'_{\bar{r},1} - \gamma'_{\bar{r},2}| \)    \left. \];
\end{align*}

(C4) $\EE[\int_0^T |f(s, 0,0,0,0,0,0,0,0)|^2 ds] <\infty.$
\begin{remark}
Note that $f(s, \cdot, \cdot,\cdot,\cdot,\cdot,\cdot,\cdot,\cdot)$ is $\mathscr{F}_s$-measurable ensures that the solution to the mean-field anticipated BSDE is $\mathscr{F}_s$-adapted.
\end{remark}

The following is the main result of this section: Two existence and uniqueness theorems for MFABSDEs and GMFDSDE, respectively.

\begin{theorem}\label{unique}
Suppose that $f$ satiesfies (C3) and (C4), and $\delta$, $\zeta$ satisfy (C1) and (C2). Then for any given terminal conditions $\xi_. \in S^2_{\mathscr{F}}(T, T+K; \mathbb{R}^m)$ and $\eta_. \in H^2_{\mathscr{F}}(T, T+K; \mathbb{R}^{m \times d})$, the mean field anticipated BSDE (\ref{eqc2}) has a unique solution, that is, there exists a unique pair of $\mathscr{F}_t$-adapted processes $(Y_., Z_.) \in S^2_{\mathscr{F}}(0, T+K; \mathbb{R}^m) \times  H^2_{\mathscr{F}}(0, T+K; \mathbb{R}^{m \times d})$ satisfying (\ref{eqc2}).
\end{theorem}
\begin{proof}
We first introduce a norm on the space $H^2_{\mathscr{F}}(0, T+K; \RR^{m}\times \RR^{m\times d})$ which is equivalent to the canonical norm:
\begin{align*}
\|v(\cdot)\|_{\beta} = \{\EE\int_0^{T+K} |v_s|^2 e^{\beta s } ds\}^{\frac{1}{2}}, \quad \beta >0 .
\end{align*}
The parameter $\beta$ will be specified later.

For any $(y,z) \in H^2_{\mathscr{F}}(0, T; \mathbb{R}^m \times \mathbb{R}^{m\times d})$, from Lemma \ref{le1}, there exists a unique solution $(Y, Z) \in S^2_{\mathscr{F}}(0, T+K; \mathbb{R}^m) \times  H^2_{\mathscr{F}}(0, T+K; \mathbb{R}^{m \times d})$ to the following anticipated BSDE:
\begin{align} \label{eq3}
Y_t = \xi &+ \int_t^T \EE^{'}\[f(s, y'_s, z'_s, y'_{s+\delta(s)}, z'_{s+\zeta(s)}, Y_s, Z_s, Y_{s+\delta(s)}, Z_{s+\zeta(s)})\]ds \nonumber  \\ 
 & -  \int_t^T Z_s dB_s,  \ \ \ \ \ \ \ \ \ \ \ \ \ \ \ \ \ \ \ \ \ \ \ \ \ \ \ \ \ \ \ \ \ \ \ \ \ \ \ \ \ \ \ \ \ \ \ \ \ \ t \in  [0, T]. 
\end{align}

Define a mapping $h:H^2_{\mathscr{F}}(0, T+K; \RR^{m}\times\RR^{m \times d}) \longrightarrow H^2_{\mathscr{F}}(0, T+K; \RR^{m}\times\RR^{m \times d})$ such that $h[(y_., z_.)] = (Y_., Z_.)$. Now we prove that $h$ is a contraction mapping under the norm $\| \cdot \|_{\beta}$. For two arbitrary elements $(y_{.,1} , z_{.,1})$ and $(y_{.,2} , z_{.,2})$ in $H^2_{\mathscr{F}}(0, T+K; \RR^{m}\times\RR^{m \times d})$, set $(Y_{.,1} , Z_{.,1}) = h[(y_{.,1} , z_{.,1})]$ and $(Y_{.,2} , Z_{.,2}) = h[(y_{.,2} , z_{.,2})]$. Denote their diferences by 
$$(\hat{y}_{.}, \hat{z}_{.}) = ((y_1 - y_2)_{.}, (z_1 - z_2)_{.}),   \quad (\hat{Y}_{.}, \hat{Z}_{.}) = ((Y_1 - Y_2)_{.}, (Z_1 - Z_2)_{.}).$$
Then, by appling Ito's formula to $e^{\beta s} |\hat{Y}_s|^2$ and by Fubini Theorem,  we get
\begin{align*}
& \EE \[\int_0^T \beta e^{\beta s} |Y_{s,1} - Y_{s,2}|^2 ds \] + \EE\[\int_0^T e^{\beta s}|Z_{s,1} - Z_{s,2}|^2 ds\]       \nonumber \\
=& \EE \int_0^T 2e^{\beta s}(Y_{s,1} - Y_{s,2}) \EE^{'}\[f\(s, y'_{1,s}, z'_{1,s}, y'_{1,s+\delta(s)}, z'_{1,s+\zeta(s)}, Y_{1,s}, Z_{1,s}, Y_{1,s+\delta(s)},   \right.  \right.    \nonumber\\
&Z_{1,s+\zeta(s)} \left. \) - f(s, y'_{2,s}, z'_{2,s}, y'_{2,s+\delta(s)}, z'_{2,s+\zeta(s)}, Y_{2,s}, Z_{2,s}, Y_{2,s+\delta(s)}, Z_{2,s+\zeta(s)})  \left. \]   ds  \nonumber \\
\le& C \EE \int_0^T 2e^{\beta s} |Y_{s,1} - Y_{s,2}|  \( |y'_{s,1} - y'_{s,2}| + |z'_{s,1} - z'_{s,2}|  + |Y_{s,1} - Y_{s,2}| \right.   \nonumber \\ 
& + |Z_{s,1} - Z_{s,2}|  +  \EE^{\mathscr{F}_s}|y'_{s+\delta(s),1} - y'_{s+\delta(s),2}| + \EE^{\mathscr{F}_s}|z'_{s+\zeta(s),1} - z'_{s+\zeta(s),2}|      \nonumber \\
& + \EE^{\mathscr{F}_s}|Y_{s+\delta(s),1} - Y_{s+\delta(s),2}| + \EE^{\mathscr{F}_s}|Z_{s+\zeta(s),1} - Z_{s+\zeta(s),2}|    \left. \)  ds   \nonumber \\
& \le  2C  \EE \int_0^{T}  e^{\beta s} |Y_{s,1} - Y_{s,2}||y_{s,1} - y_{s,2}| ds \nonumber \\
&+ 2C  \EE \int_0^{T}  e^{\beta s} |Y_{s,1} - Y_{s,2}| |z_{s,1} - z_{s,2}| ds  \nonumber \\
& + 2C  \EE \int_0^{T}  e^{\beta s} |Y_{s,1} - Y_{s,2}| |y_{s+\delta(s),1} - y_{s+\delta(s),2}| ds   \nonumber \\
&+ 2C \EE\int_0^{T} e^{\beta s}  |Y_{s,1} - Y_{s,2}| |Z_{s,1} - Z_{s,2}|ds+ 2C\EE \int_0^{T} e^{\beta s}  |Y_{s,1} - Y_{s,2}|^2 ds   \nonumber \\
&+ 2C  \EE \int_0^{T}  e^{\beta s} |Y_{s,1} - Y_{s,2}||z_{s+\zeta(s),1} - z_{s+\zeta(s),2}| ds   \nonumber \\
& +  2C\EE \int_0^{T} e^{\beta s}  |Y_{s,1} - Y_{s,2}|  |Y_{s+\delta(s),1} - Y_{s+\delta(s),2}|ds   \nonumber \\
&+ 2C\EE \int_0^{T} e^{\beta s}  |Y_{s,1} - Y_{s,2}| |Z_{s+\zeta(s),1} - Z_{s+\zeta(s),2}| ds    \nonumber \\
&= (1)+(2)+(3)+(4)+(5)+(6)+(7)+(8).
\end{align*}
Since
\begin{align*}
(1) &\le C \EE \int_0^{T+K} e^{\beta s} \( \frac{\beta}{8C} |Y_{s,1} -Y_{s,2}|^2 + \frac{8C}{\beta}|y_{s,1} - y_{s,2}|^2 \) ds   \nonumber \\
&= \frac{\beta}{8} \EE \int_0^{T+K} e^{\beta s}  |Y_{s,1} - Y_{s,2}|^2 ds + \frac{8C^2}{\beta} \EE \int_0^{T+K} e^{\beta s} |y_{s,1} - y_{s,2}|^2  ds,
\end{align*}
\begin{align*}
(2) &\le C \EE \int_0^{T+K} e^{\beta s} \( \frac{\beta}{8C} |Y_{s,1} -Y_{s,2}|^2 + \frac{8C}{\beta}|z_{s,1} - z_{s,2}|^2 \) ds   \nonumber \\
&= \frac{\beta}{8} \EE \int_0^{T+K} e^{\beta s}  |Y_{s,1} - Y_{s,2}|^2 ds + \frac{8C^2}{\beta} \EE \int_0^{T+K} e^{\beta s} |z_{s,1} - z_{s,2}|^2  ds,
\end{align*}
\begin{align*}
(3) &\le C \EE \int_0^{T+K} e^{\beta s} \( \frac{\beta}{8C} |Y_{s,1} -Y_{s,2}|^2 + \frac{8C}{\beta} L \cdot |y_{s,1} - y_{s,2}|^2 \) ds   \nonumber \\
&= \frac{\beta}{8} \EE \int_0^{T+K} e^{\beta s}  |Y_{s,1} - Y_{s,2}|^2 ds + \frac{8C^2 L}{\beta} \EE \int_0^{T+K} e^{\beta s} |y_{s,1} - y_{s,2}|^2  ds,
\end{align*}
\begin{align*}
(4) &\le C \EE \int_0^{T+K} e^{\beta s} \( 4C |Y_{s,1} -Y_{s,2}|^2 + \frac{1}{4C}|Z_{s,1} - Z_{s,2}|^2 \) ds   \nonumber \\
&= 4C^2 \EE \int_0^{T+K} e^{\beta s} |Y_{s,1} - Y_{s,2}|^2 ds + \frac{1}{4} \EE \int_0^{T+K} e^{\beta s} |Z_{s,1} - Z_{s,2}|^2  ds,
\end{align*}
\begin{align*}
(6) &\le C \EE \int_0^{T+K} e^{\beta s} \( \frac{\beta}{8C} |Y_{s,1} -Y_{s,2}|^2 + \frac{8CL}{\beta}|z_{s,1} - z_{s,2}|^2 \) ds   \nonumber \\
&= \frac{\beta}{8} \EE \int_0^{T+K} e^{\beta s}  |Y_{s,1} - Y_{s,2}|^2 ds + \frac{8C^2L}{\beta} \EE \int_0^{T+K} e^{\beta s} |z_{s,1} - z_{s,2}|^2  ds,
\end{align*}
\begin{align*}
(7)  & \le  C \EE \int_0^{T+K} e^{\beta s}  |Y_{s,1} - Y_{s,2}|^2 ds + CL \EE \int_0^{T+K} e^{\beta s} |Y_{s,1} - Y_{s,2}|^2  ds    \nonumber \\
& = C(1+L) \EE \int_0^{T+K} e^{\beta s}  |Y_{s,1} - Y_{s,2}|^2 ds,
\end{align*}
\begin{align*}
(8) &\le C \EE \int_0^{T} e^{\beta s} \( 4CL  |Y_{s,1} -Y_{s,2}|^2 + \frac{1}{4CL}|Z_{s +\zeta(s),1} - Z_{s +\zeta(s),2}|^2 \) ds   \nonumber \\
& \le  4C^2L \EE \int_0^{T+K} e^{\beta s}  |Y_{s,1} - Y_{s,2}|^2 ds + \frac{1}{4L} \EE \int_0^{T+K} e^{\beta s} L \cdot |Z_{s,1} - Z_{s,2}|^2  ds    \nonumber \\
& = 4C^2L \EE \int_0^{T+K} e^{\beta s}  |Y_{s,1} - Y_{s,2}|^2 ds + \frac{1}{4}  \EE \int_0^{T+K} e^{\beta s}  |Z_{s,1} - Z_{s,2}|^2 ds,
\end{align*}
consequently,
\begin{align*}
& \( \frac{\beta}{2}- 2C - 4C^2 - C(1+L) - 4C^2L \)\EE \[\int_0^{T+K} e^{\beta s}   |\hat{Y}_s|^2 ds \]  \nonumber \\
&\ \ \ \ \ + \frac{1}{2}\EE\[\int_0^{T+K} e^{\beta s}|\hat{Z}_s|^2 ds\]       \le \frac{8C^2(L+1)}{\beta} \EE \[\int_0^{T+K} e^{\beta s}  \( |\hat{y}_s|^2 + |\hat{z}_s|^2\) ds \].
\end{align*}
We choose $\beta = 32C^2L + 32C^2 + 6C + 2CL +1$, such that
\begin{align*}
& \EE \[\int_0^{T+K}\(  |\hat{Y}_s |^2 + |\hat{Z}_s |^2 \) e^{\beta s}  ds \]       
\le  \frac{1}{2} \EE \[\int_0^{T+K}   \( |\hat{y}_s|^2 + |\hat{z}_s|^2\) e^{\beta s} ds\].
\end{align*}
Thus, $h$ is a contraction, and hence the conclusions of the theorem follows from Schauder's fixed point theorem.
\end{proof}

In addition to the above existence and uniqueness thoerem for (\ref{eqc2}), we need to prove the same theorem for the following generalized mean-field delay stochastic differential equations(GMFDSDE):
\begin{eqnarray}\label{MFDSDE}
\begin{cases}
dX_t =b\(t, X_t, X_{t-\delta}, P_{X_t}, P_{X_{t-\delta}}\)dt +\sigma \(t, X_t, X_{t-\delta}, P_{X_t}, P_{X_{t-\delta}}\) dB_t  ,  \\ 
\ \ \ \ \ \ \ \ \ \ \ \ \ \ \ \ \ \ \ \ \ \ \ \ \ \ \ \ \ \ \ \ \ \ \ \ \ \ \ \ \ \ \ \ \ \ \ \ \ \ \ \ \ \ \ \ \ \ \ \ \ \ \ \ \ \ \ \ \ \ \ \ \ \ \ \   t \in [0,T];
\cr 
X_t = \xi_t, \ \ \ \ \ \ \ \ \ \ \ \ \ \ \ \ \ \ \ \ \ \ \ \ \ \ \ \ \ \ \ \ \ \ \ \ \ \ \ \ \ \ \ \ \ \ \ \ \ \ \ \ \ \ \ \ \ \ \ \ \ \ \ \ t \in  [-\delta, 0)
\end{cases}
\end{eqnarray}
which will be applied to the control problem. 

Assume that for all $t \in [-\delta, T]$, $b: [-\delta, T] \times \RR^m \times \RR^m \times \mathcal{P}_2(\RR^m)\times \mathcal{P}_2(\RR^m)\to H^2(-\delta, T; \RR^m)$, satisfies the following conditions:

(C5) There exists a constant $C \ge 0$, such that for all $t \in [-\delta, T]$, $x, x', x_{\delta}, x'_{\delta} \in \RR^{m}$, $\mu, \mu_{\delta}, \mu', \mu'_{\delta} \in  \mathcal{P}_2(\RR^m)$, we have 
\begin{align*}
& |b\(t, x, x_{\delta}, \mu, \mu_{\delta}\)- b\(t, x', x'_{\delta}, \mu', \mu'_{\delta}\)|  \\
& \le C\(|x - x'| + |x_{\delta} - x'_{\delta}| + W_2(\mu, \mu') + W_2(\mu_{\delta}, \mu'_{\delta}) \)
\end{align*}
$\sigma$ satisfies the same condition as $b$.

(C6) $\sup \limits_{t \ge -\delta} \(|b(t, 0, 0, 0, 0)| + |\sigma(t, 0, 0, 0, 0)|\) < \infty. $
 
\begin{theorem}\label{TMFDSDE}
Suppose that $b$ and $\sigma$ satiesfies (C5) and (C6), then for any given delay conditions $\xi_. \in H^2_{\mathscr{F}}(-\delta, 0; \mathbb{R}^m)$, the MFDSDE (\ref{MFDSDE})
has a unique strong solution.
\end{theorem}
\begin{proof}
For any $\beta \ge 0$, we introduce  a norm in the Banach space $H^2_{\mathscr{F}}(-\delta, T; \RR^{m})$:
$$
\| \nu(\cdot) \|_{\beta} = \( \EE \[\int_{-\delta}^T e^{- \beta t} | \nu _s |^2 dt \] \)^2
$$
Clearly, it is equivalent to the original norm. We use this norm to construct a contraction mapping that allow us to apply the fixed point Theorem. Set 
\begin{eqnarray}
\begin{cases}
X_t = \xi_0 + \int_0^t b\(s, X_s, X_{s-\delta}, P_{X_s}, P_{X_{s-\delta}}\)ds + \int_0^s \sigma \(s, X_s, X_{s-\delta}, P_{X_s}, P_{X_{s-\delta}}\) dB_s  , \nonumber \\  
\ \ \ \ \ \ \ \ \ \ \ \ \ \ \ \ \ \ \ \ \ \ \ \ \ \ \ \ \ \ \ \ \ \ \ \ \ \ \ \ \ \ \ \ \ \ \ \ \ \ \ \ \ \ \ \ \ \ \ \ \ \ \ \ \ \ \ \ \ \ \ \ \ \ \ \ \ \ \ t \in [0,T];
\cr 
X_t = \xi_t,  \ \ \ \ \ \ \ \ \ \ \ \ \ \ \ \ \ \ \ \ \ \ \ \ \ \ \ \ \ \ \ \ \ \ \ \ \ \ \ \ \ \ \ \ \ \ \ \ \ \ \ \ \ \ \ \ \ \ \ \ \ \ \ \ \ \ t \in  [-\delta, 0). \nonumber
\end{cases}
\end{eqnarray}
Given $x \in H^2_{\mathscr{F}}(-\delta, T; \RR^{m})$, we define
\begin{eqnarray}
\begin{cases}
X_t = \xi_0 + \int_0^t b\(s, x_s, x_{s-\delta}, P_{x_s}, P_{x_{s-\delta}}\)ds + \int_0^s \sigma \(s, x_s, x_{s-\delta}, P_{x_s}, P_{x_{s-\delta}}\) dB_s  , \nonumber \\  
\ \ \ \ \ \ \ \ \ \ \ \ \ \ \ \ \ \ \ \ \ \ \ \ \ \ \ \ \ \ \ \ \ \ \ \ \ \ \ \ \ \ \ \ \ \ \ \ \ \ \ \ \ \ \ \ \ \ \ \ \ \ \ \ \ \ \ \ \ \ \ \ \ \ \ \ \ \ \ t \in [0,T];
\cr 
X_t = \xi_t,  \ \ \ \ \ \ \ \ \ \ \ \ \ \ \ \ \ \ \ \ \ \ \ \ \ \ \ \ \ \ \ \ \ \ \ \ \ \ \ \ \ \ \ \ \ \ \ \ \ \ \ \ \ \ \ \ \ \ \ \ \ \ \ \ \ \ t \in  [-\delta, 0). \nonumber
\end{cases}
\end{eqnarray}
Then, $X \in H^2_{\mathscr{F}}(-\delta, T; \RR^{m})$. Denote $X = \Phi (x.)$, now we prove that $\Phi$ is a contraction mapping under the norm $\| \cdot \|_{\beta}$. For two arbitrary elements $x.$ and $x'.$ in $H^2_{\mathscr{F}}(-\delta, T; \RR^{m})$, set $x. = \Phi(x.)$, $x'. = \Phi(x'.)$. Denote their differences by $\bar{x} = x - x'$, $\bar{\Phi} = \Phi(x) -\Phi(x')$, $\bar{b} = b(t, x_t, x_{t-\delta}, P_{x_t}, P_{x_{t-\delta}}) - b(t, x'_t, x'_{t-\delta}, P_{x'_t}, P_{x'_{t-\delta}})$, $\bar{\sigma} = \sigma(t, x_t, x_{t-\delta}, P_{x_t}, P_{x_{t-\delta}}) - \sigma(t, x'_t, x'_{t-\delta}, P_{x'_t}, P_{x'_{t-\delta}})$ .

By applying Ito's formula to $e^{\beta t} \bar{\Phi}_t^2$, we have
\begin{align*}
&e^{\beta t} \bar{\Phi}_t^2 + \beta \int_0^T e^{-\beta t} \bar{\Phi}_t^2 dt \\
=& 2\int_0^T e^{-\beta t} \langle \bar{\Phi}_t, \bar{b} \rangle dt + 2\int_0^T e^{-\beta t} \langle \bar{\Phi}_t, \bar{\sigma} \rangle dB_t + \int_0^T e^{-\beta t} tr \langle \bar{\sigma}_t, \bar{\sigma}^* \rangle dt
\end{align*}
Take expectation on both sides, then 
\begin{align*}
 \beta \EE \int_0^T e^{-\beta t} \bar{\Phi}_t^2 dt = 2 \EE \int_0^T e^{-\beta t} \langle \bar{\Phi}_t, \bar{b} \rangle dt + \EE \int_0^T e^{-\beta t} tr \langle \bar{\sigma}_t, \bar{\sigma}^* \rangle dt,
\end{align*}
by Cauchy-Schwartz enequality and condition (C5), 
\begin{align*}
 \beta \EE \int_0^T e^{-\beta t} \bar{\Phi}_t^2 dt \le \EE \int_0^T e^{-\beta t} \bar{\Phi}_t^2 dt + 2 C^2 \EE \int_0^T e^{-\beta t} |\bar{x}_t|^2 dt ,
\end{align*}
we choose $\beta = 1+ 4C^2$, such that 
\begin{align*}
\EE \int_0^T e^{-\beta t} \bar{\Phi}_t^2 dt \le \frac{1}{2} \EE \int_0^T e^{-\beta t} |\bar{x}_t|^2 dt.
\end{align*}
Consequently, $\Phi$ is a strict contraction mapping, which complete the proof.
\end{proof}

\subsection{Comparison theorem for MFABSDEs}
Notice that the conditions on the driver $f$ which is needed for the comparison theorem for mean-field BSDEs and for the anticipated BSDEs are stronger than those needed for the existence and uniqueness theorem. 

Let $\bar{f}_i = (t, y, z, y', z'), i = 1, 2,$ be two drivers of mean-field BSDEs, to derive the comparison principle for mean-field BSDEs, restrictions are forced on $\bar{f}_i, i = 1, 2,$ in \cite{BLP2009} as following:
\begin{enumerate}[(i)]
\setcounter{enumi} {0}
\item  One of the two coefficients is independent of $z'$,

\item One of the two coefficients is nondecreasing in $y'$.
\end{enumerate}

On the other hand, two example in \cite{PY} also given to demonstrate the comparison principle for the anticipated BSDEs (\ref{eqq4}). Let $\hat{f}_i = (t, y, z, \theta, \gamma), i = 1, 2,$ be two drivers of (\ref{eqq4}), if
\begin{enumerate}[(i)]
\setcounter{enumi} {2}
\item  $\hat{f}_2$ is increasing in the anticipated term of $Y.$

\item $\hat{f}_2$ indipendent of the anticipated term of $Z.$
\end{enumerate}
then the comparison theorem holds for anticipated BSDEs.

Now we discuss the comparison principle for mean-filed anticipated BSDEs (\ref{eqc2}), it is naturally to combine all the restrictions above both on $\bar{f}$ and $\hat{f}$. In addition, we force the other two restrictions on $f$:
\begin{enumerate}[(i)]
\setcounter{enumi} {4}
\item  One of the two coefficients ($f_1$ or $f_2$) is independent of the anticipated term of $z'$,

\item $f_2$ is non-decreasing in the anticipated term of $Y'.$
\end{enumerate}

Counterexample are given to show that if the driver $f$ of mean-field anticipated BSDEs depends on the anticipated term of $z'$ we can't get the comparison theorem.
\begin{example}
For $d = 1$ we consider the mean-field BSDE(\ref{eqc2}) with time horizon $T = 1$, for all $t \in [0, T]$, with driver $f(t, y'_t, z'_t, y'_{t+\delta(t)}, z'_{t+\zeta(t)}, y_t, z_t, y_{t+\delta(t)},$\\$ z_{t+\zeta(t)}) = - z'_{t+\zeta(t)}$ and two different terminal values $\xi_1, \xi_2 \in L^2(\Omega, \mathcal{F}_T, P)$. Let us denotes the associated solutions by $(Y^1, Z^1)$ and $(Y^2, Z^2)$, respectively. Then,
\begin{eqnarray} \label{eqEoCT}
Y_t^i = \xi_i + \int_t^1 \EE[- Z_{s + \zeta(s)}^i]ds - \int_t^1 Z_s^i dW_s, \quad 0 \le t \le 1, i = 1, 2.
\end{eqnarray}
Let $\xi_1 = - (B_1^{+})^3$, then by the Clark-Ocone formula we have:
\begin{eqnarray}
- (B_1^{+})^3 = \xi_1 &=& \EE(\xi) + \int_0^1 \EE (D_s \xi | \mathscr{F}_s) dW_s  \nonumber \\
&=& \EE(\xi) + \int_0^1 Z_s dW_s = -\frac{2}{\sqrt{2 \pi}} - \int_0^1 3(B_1^{+})^2 dW_s. \nonumber 
\end{eqnarray}
In addition, 
\begin{eqnarray}
Y_0^1 &=& - (B_1^+)^3 - \int_0^1 \( \EE(D_{s+\zeta(s)}\xi)1_{s + \zeta(s)\le 1} + \EE(D_1 \xi)1_{s + \zeta(s) >1} \) ds - \int_0^1 Z_sdW_s \nonumber \\
&=& -\frac{2}{\sqrt{2 \pi}} + \int_0^1 3 \EE (B_1^+)^2 ds = -\frac{2}{\sqrt{2 \pi}} + \frac{3}{2} >0. \nonumber
\end{eqnarray}
Let now $\xi_2 = 0$. Then $(Y^2, Z^2) = (0, 0)$ is also a solution of (\ref{eqEoCT}). Hence, we have $Y_0^1 > Y_0^2$ although $\xi_1 \le \xi_2$.
\end{example}

Without loss of generality, we assume that all the restrictions (i) - (vi) are satisfied by $f_2$.
Let $(Y.^{(1)}, Z.^{(1)})$, $(Y.^{(2)}, Z.^{(2)})$ be respectively solutions of the following two mean-filed anticipated BSDEs:
\begin{eqnarray}
\begin{cases}
Y_t^{(1)} =\xi_T^{(1)} + \int_t^T  \EE^{'}[f_1(s,  Y^{(1)}_s, Z^{(1)}_s, Y^{(1)}_{s+\delta(s)}, Z^{(1)}_{s+\zeta(s)}, Y'^{,(1)}_s, Z'^{,(1)}_s,Y'^{,(1)}_{s+\delta(s)}, Z'^{,(1)}_{s+\zeta(s)})]ds \nonumber \\
\ \ \ \ \ \ \ \ \ \ \ \ \ \ \  - \int_t^T Z^{(1)}_sdB_s  ,  \ \ \ \ \ \ \ \ \ \ \ \ \ \ \ \ \ \ \ \ \ \ \ \ \ \ \ \ \ \ \ \ \ \ \ \ \ \ \ \ \ \ \ \ \ \ \ \ \ \ \ \  t \in [0,T];
\cr 
Y_t^{(1)} = \xi_t^{(1)},  Z_t^{(1)} = \eta_t^{(1)} ,  \ \ \ \ \ \ \ \ \ \ \ \ \ \ \ \ \ \ \ \ \ \ \ \ \ \ \ \ \ \ \ \ \ \ \ \ \ \ \ \ \ \ \ \ \ \ \ \   t \in  (T, T+K];
\end{cases}
\end{eqnarray}
\begin{eqnarray}
\begin{cases}
Y_t^{(2)} =\xi_T^{(2)} + \int_t^T  \EE^{'}[f_2(s, Y^{(2)}_s, Z^{(2)}_s, Y^{(2)}_{s+\delta(s)},Y'^{,(2)}_s, Y'^{,(2)}_{s+\delta(s)})]ds - \int_t^T Z^{(2)}_sdB_s,   \\   
\ \ \ \ \ \ \ \ \ \ \ \ \ \ \ \ \ \ \ \ \ \ \ \ \ \ \ \ \ \ \ \ \ \ \ \ \ \ \ \ \ \ \ \ \ \ \ \ \ \ \ \ \ \ \ \ \ \ \ \ \ \ \ \ \ \ \ \ \ \ \ \ \ \ \ \ \ \ \ \ \ \ \ \ \ \ t \in [0,T];
\cr 
Y^{(2)}_t = \xi_t^{(2)}, \ \ \ \ \ \ \ \ \ \ \ \ \ \ \ \ \ \ \ \ \ \ \ \ \ \ \ \ \ \ \ \ \ \ \ \ \ \ \ \ \ \ \ \ \ \ \ \ \ \ \ \ \ \ \ \ \ \ \ \ \ \ \ t \in  (T, T+K]; \nonumber
\end{cases}
\end{eqnarray}

\begin{theorem}
Assume that $f_2$ the above restrictions (i)-(vi), $\xi_{.}^{(1)}$, $\xi_{.}^{(2)} \in \mathcal{S}^2_{\mathscr{F}}(T, T+K)$, $\delta$, $\zeta$ satisfies (C1), (C2), and for all $t \in [0, T]$, $y \in \mathbb{R}^m$, $z \in \mathbb{R}^{m \times d}$, $f_2(t, y, z, \cdot, y', \cdot)$ is increasing, that is, $f_2(t, y, z, \theta_r, y', \theta'_r) \ge f_2(t, y, z, \tilde{\theta}_r, y', \tilde{\theta}'_r)$, if $\theta_r \ge \tilde{\theta}_r$ and $\theta'_r \ge \tilde{\theta}'_r$, $\theta_r, \theta'_r, \tilde{\theta}_r, \tilde{\theta}'_r \in H^2_{\mathscr{F}}(t, T+K), r \in [t, T+K]$. If $\xi_s^{(1)} \ge \xi_s^{(2)}, s \in [T, T+K]$ and $f_1(t, y, z, \theta_r, \gamma_{\bar{r}}, y', z', \theta'_r, \gamma'_{\bar{r}}) \ge f_2(t, y, z, \theta_r, y', \theta'_r), r, \bar{r} \in [t, T+K]$, then $$Y_t^{(1)} \ge Y_t^{(2)}, \quad \quad a.e., a.s.$$
\end{theorem}
\begin{proof}
Set
\begin{eqnarray}
\begin{cases}
Y_t^{(3)} =\xi_T^{(2)} + \int_t^T  \EE^{'}[f_2(s, Y^{(3)}_s, Z^{(3)}_s, Y^{(1)}_{s+\delta(s)},Y'^{,(3)}_s, Y'^{,(1)}_{s+\delta(s)})]ds - \int_t^T Z^{(3)}_sdB_s  , \\
\ \ \ \ \ \ \ \ \ \ \ \ \ \ \ \ \ \ \ \ \ \ \ \ \ \ \ \ \ \ \ \ \ \ \ \ \ \ \ \ \ \ \ \ \ \ \ \ \ \ \ \ \ \ \ \ \ \ \ \ \ \ \ \ \ \ \ \ \ \ \ \ \ \ \ \ \ \ \ \ \ \ \ \ \ t \in [0,T];
\cr 
Y_t^{(3)} = \xi_t^{(2)}, \ \ \ \ \ \ \ \ \ \ \ \ \ \ \ \ \ \ \ \ \ \ \ \ \ \ \ \ \ \ \ \ \ \ \ \ \ \ \ \ \ \ \ \ \ \ \ \ \ \ \ \ \ \ \ \ \ \ \ \ \ \ \ t \in  (T, T+K].\nonumber
\end{cases}
\end{eqnarray}
By Lemma \ref{le1}, we know that there exists a unique pair of $\mathscr{F}_t$-adapted processes $(Y.^{(3)}, Z.^{(3)}) \in S_{\mathscr{F}}^2(0, T+K, \RR^m) \times H_{\mathscr{F}}^2(0, T; \mathbb{R}^{m \times d})$ that satisfies the above BSDE. Since $f_1 \ge f_2$, $y \in \mathbb{R}^m, z \in \mathbb{R}^{m \times d}$, by Lemma \ref{le2}, we obtain 
\begin{align*}
Y_t^{(1)} \ge Y_t^{(3)}, \quad a.s.
\end{align*}
Set
\begin{eqnarray}
\begin{cases}
Y_t^{(4)} =\xi_T^{(2)} + \int_t^T  \EE^{'}[f_2(s, Y^{(4)}_s, Z^{(4)}_s, Y^{(3)}_{s+\delta(s)},Y'^{,(4)}_s, Y'^{,(3)}_{s+\delta(s)})]ds - \int_t^T Z^{(4)}_sdB_s  , \\
\ \ \ \ \ \ \ \ \ \ \ \ \ \ \ \ \ \ \ \ \ \ \ \ \ \ \ \ \ \ \ \ \ \ \ \ \ \ \ \ \ \ \ \ \ \ \ \ \ \ \ \ \ \ \ \ \ \ \ \ \ \ \ \ \ \ \ \ \ \ \ \ \ \ \ \ \ \ \ \ \ \ \ \ \ t \in [0,T];
\cr 
Y_t^{(4)} = \xi_t^{(2)}, \ \ \ \ \ \ \ \ \ \ \ \ \ \ \ \ \ \ \ \ \ \ \ \ \ \ \ \ \ \ \ \ \ \ \ \ \ \ \ \ \ \ \ \ \ \ \ \ \ \ \ \ \ \ \ \ \ \ \ \ \ \ \ t \in  (T, T+K].\nonumber
\end{cases}
\end{eqnarray}
Since for all $t \in [0, T], y \in \mathbb{R}^m, z \in \mathbb{R}^{m \times d}$, $f_2(t, Y_t, Z_t, \cdot, Y'_t, Z'_s, Y'_{s+\delta(s)}, Z'_{t+\zeta(t)})$ is increasing and $Y_t^{(1)} \ge Y_t^{(3)}$, by Lemma \ref{le2}, we know 
\begin{align*}
Y_t^{(3)} \ge Y_t^{(4)}, \quad a.s.
\end{align*}
For $n = 5, 6, \cdot\cdot\cdot,$ we consider the following mean-field BSDE:
\begin{eqnarray}
\begin{cases}
Y_t^{(n)} =\xi_T^{(2)} + \int_t^T  \EE^{'}[f_2(s, Y^{(n)}_s, Z^{(n)}_s, Y^{(n-1)}_{s+\delta(s)},Y'^{,(n)}_s, Y'^{,(n-1)}_{s+\delta(s)})]ds\\
 \ \ \ \ \ \ \ \ \ \ \ \ \ \ \  - \int_t^T Z^{(n)}_sdB_s  , \ \ \ \ \ \ \ \ \ \ \ \ \ \ \ \ \ \ \ \ \ \ \ \ \ \ \ \ \ \ \ \ \ \   t \in [0,T];
\cr 
Y_t^{(n)} = \xi_t^{(2)}, \ \ \ \ \ \ \ \ \ \ \ \ \ \ \ \ \ \ \ \ \ \ \ \ \ \ \ \ \ \ \ \ \ \ \ \ \ \ \ \ \ \ \ \ \ t \in  (T, T+K]. \nonumber
\end{cases}
\end{eqnarray}
Similarly, we have $Y_t^{(4)} \ge Y_t^{(5)} \ge \cdot\cdot\cdot \ge Y_t^{(n)} \ge \cdot\cdot\cdot,$ a.s. We use $\|\nu(\cdot)\|_{\beta}$ in the proof of Theorem \ref{unique} as the norm in the Banach space $H^2_{\mathscr{F}}(0, T+K; \RR^m) \times H^2_{\mathscr{F}}(0, T; \RR^{m\times d})$. Set $\hat{Y}_s^{(n)}= Y_s^{(n)} - Y_s^{(n-1)}$, $\hat{Z}_s^{(n)}= Z_s^{(n)} - Z_s^{(n-1)}$, $n \ge 4.$
Then by (\ref {eatimates2}), we have
\begin{align}
& \EE\[ \int_0^T \( \frac{\beta}{2}  |\hat{Y}^{(n)}_{s}|^2  +  |\hat{Z}^{(n)}_{s}|^2 \)  e^{\beta s} ds \] \nonumber \\
\le &\frac{2}{\beta} \EE\[ \int_0^T  \EE' |f_2(s, Y^{(n)}_s, Z^{(n)}_s, Y^{(n-1)}_{s+\delta(s)},Y'^{,(n)}_s, Y'^{,(n-1)}_{s+\delta(s)}) \right. \nonumber\\
&\qquad  \qquad  \quad   - f_2(s, Y^{(n-1)}_s, Z^{(n-1)}_s, Y^{(n-2)}_{s+\delta(s)},Y'^{,(n-1)}_s, Y'^{,(n-2)}_{s+\delta(s)}) |^2  e^{\beta s}ds\left. \]   \nonumber \\
\le &\frac{10C^2}{\beta} \EE \[ \int_0^T  \(2 | \hat{Y}_s^{(n)}|^2 + |\hat{Z}_s^{(n)}|^2 + |\hat{Y}_{s+\delta{s}}^{(n-1)}|^2 + | \hat{Z}_{s+\delta{s}}^{(n-1)}|^2 \) e^{\beta s} ds \] \nonumber  \\
\le& \frac{20C^2}{\beta} \EE \[ \int_0^T  \( | \hat{Y}_s^{(n)}|^2 + |\hat{Z}_s^{(n)}|^2 \) e^{\beta s} ds \] + \frac{10C^2L}{\beta} \EE \[ \int_0^T  \( |\hat{Y}_{s}^{(n-1)}|^2 + | \hat{Z}_{s}^{(n-1)}|^2\) e^{\beta s} ds \]. \nonumber 
\end{align}
Set $\beta = 60C^2L + 60C^2 +3$. Then 
\begin{align}
\frac{2}{3}\EE\[ \int_0^T \(  |\hat{Y}^{(n)}_{s}|^2 +  |\hat{Z}^{(n)}_{s}|^2 \)  e^{\beta s} ds \] \le  \frac{1}{3}\EE \[ \int_0^T  \( |\hat{Y}_{s}^{(n-1)}|^2 + | \hat{Z}_{s}^{(n-1)}|^2\) e^{\beta s} ds \]   \nonumber
\end{align}
Hence,
\begin{align}
\EE\[ \int_0^T \(  |\hat{Y}^{(n)}_{s}|^2 +  |\hat{Z}^{(n)}_{s}|^2 \)  e^{\beta s} ds \] \le  \(\frac{1}{2}\)^{n-4}\EE \[ \int_0^T  \( |\hat{Y}_{s}^{(4)}|^2 + | \hat{Z}_{s}^{(4)}|^2\) e^{\beta s} ds \].   \nonumber
\end{align}
It follows that $\(Y.^{(n)}\)_{n \ge 4}$ and $\(Z.^{(n)}\)_{n \ge 4}$ are Cauchy sequences in $H^2_{\mathscr{F}}(0, T+K; \RR^m) \times H^2_{\mathscr{F}}(0, T; \RR^{m\times d})$. Denote their limits by $Y.$ and $Z.$, respectively. Since $H^2_{\mathscr{F}}(0, T+K; \RR^m)$ and $H^2_{\mathscr{F}}(0, T; \RR^{m\times d})$ are both Banach spaces, we obtain $(Y. \times Z.) \in H^2_{\mathscr{F}}(0, T+K; \RR^m) \times H^2_{\mathscr{F}}(0, T; \RR^{m\times d})$. Note that for all $t \in[0, T]$,
\begin{align}
&\EE\[ \int_t^T  \EE' |f_2(s, Y^{(n)}_s, Z^{(n)}_s, Y^{(n-1)}_{s+\delta(s)},Y'^{,(n)}_s, Y'^{,(n-1)}_{s+\delta(s)}) - f_2(s, Y_s, Z_s, Y_{s+\delta(s)},Y'_s, Y'_{s+\delta(s)}) |^2  e^{\beta s}ds \]   \nonumber \\
\le&  5C^2  \EE \[ \int_0^T  \(2 | Y_s^{(n)} - Y_s|^2 + |Z_s^{(n)} - Z_s|^2 + L| Y_s^{(n-1)} - Y_s|^2 + L| Z_s^{(n-1)} - Z_s|^2 \) e^{\beta s} ds \] \to 0, \nonumber
\end{align}
when $n \to \infty$. Therefore, $(Y., Z.)$ satisfies the following mean-field anticipated BSDE:
\begin{eqnarray}
\begin{cases}
Y_t =\xi_T^{(2)} + \int_t^T  \EE^{'}[f_2(s, Y_s, Z_s, Y_{s+\delta(s)},Y'_s, Y'_{s+\delta(s)})]ds- \int_t^T Z_sdB_s, \ \ t \in [0,T];
\cr 
Y_t = \xi_t^{(2)},\ \ \ \ \ \ \ \ \ \ \ \ \ \ \ \ \ \ \ \ \ \ \ \ \ \ \ \ \ \ \ \ \ \ \ \ \ \ \ \ \ \ \ \ \ \ \ \ \ \ \ \ \ \ \ \ \ \ \ \ \ \ \ \ \ \ \ t \in  (T, T+K]. \nonumber
\end{cases}
\end{eqnarray}
By Theorem \ref{unique}, we know
\begin{align*} 
Y_t = Y_t^{(2)}, \quad a.s.
\end{align*}
Since $Y_t^{(1)} \ge Y_t^{(3)} \ge Y_t^{(4)} \ge Y_t$, it holds immediately
\begin{align*} 
Y_t^{(1)}\geq Y_t^{(2)}, \quad a.s.
\end{align*}
\end{proof}

\section{Formulation of the generalized mean-filed stochastic delay control problem}
\setcounter{equation}{0}

\indent In this section, we give the formulation of our generalized mean-field optimal control problem.\\ 
\indent We consider the generalized mean-field delay type optimal control system, with the state equation (\ref{2eq11})
and the cost functional (\ref{2eq12}). From Theorem \ref{TMFDSDE}, we know equation (\ref{2eq11}) admits a unique solution.
Recall that the agent wishes to minimize his cost functional, namely, an admissible control $u\in \mathcal{U}$ is said to be optimal if
$$J(u)=\min_{v\in \mathcal{U}}J(v).$$\
\indent Throughout this paper, we make the following assumptions on the coefficients:
\begin{hypothesis}\label{2H1}
(1) The given functions $b, \sigma, h ,\Phi$ are differentiable with respect to $(x,x_{\delta},\mu,\mu_\delta, v,v_\delta).$\\
\indent (2) $b, \sigma$ are Lipschitz continuous w.r.t. $(x,x_\delta,\mu,\mu_\delta)$, The derivatives of $b, \sigma$ are Lipschitz continuous and bounded.\\
\indent (3)  The derivatives of $h, \Phi$ are Lipschitz continuous and bounded by $C(1+|x|+|x_\delta|+|v|+|v_\delta| )$.\\
\end{hypothesis}

\indent We will make use of the following notations concerning matrices. We denote by $\mathbb{R}^{n\times d}$ the space of real matrices of $n\times d $-type, and by $\mathbb{R}^{n\times n}_d$ the linear space of the vectors of matrices $M=(M_1,\cdots, M_d)$, with $M_i\in \mathbb{R}^{n\times n}$, $1\leq i \leq d.$ Given any $\alpha, \beta\in \mathbb{R}^{n}$,  $L, S \in \mathbb{R}^{n\times d}$,  $\gamma \in \mathbb{R}^{d}$ and  $M, N \in \mathbb{R}^{n\times n}_d$, we introduce the following notation: $\alpha \beta=\sum^{n}_{i=1}\alpha_i \beta_i \in \mathbb{R}$, $\alpha\times\beta=(\alpha_i \beta_j)_{1\leq i,j\leq n}$; $LS=\sum^{d}_{i=1} L_i S_i \in \mathbb{R}$, where $L=(L_1,\cdots, L_d), S=(S_1,\cdots, S _d)$;
$ML=\sum^{d}_{i=1}M_i L_i \in \mathbb{R}^{n}$; $M\alpha \gamma=\sum^{d}_{i=1}(M_i \alpha) \gamma_i \in \mathbb{R}^n$; $MN=\sum^{d}_{i=1}M_i N_i \in \mathbb{R}^{n\times n}$; For simplicity, we use the following notations $$\Theta_t=(X^u_t,X^u_{t-\delta}, P_ {{X^u_t}},P_ {{X^u_{t-\delta}}},u_t,u_{t-\delta})$$; $$\Theta'_t=((X^u_t)',(X^u_{t-\delta})', P_ {{X^u_t}},P_ {{X^u_{t-\delta}}},(u_t)',(u_{t-\delta})').$$

\indent Let us suppose that $u$ is an optimal control and $X^u$ the associated optimal trajectory. Then we introduce the convex perturbed control as follows:
$$u^{\theta}_t=u_t+\theta(v_t-u_t),$$
where $\theta\geq 0$ is sufficiently small, and $v_t$ is an arbitrary element of $\mathcal{U}$, $X^{\theta}$ is the state under the control $u^{\theta}$. The convexity of $U$ guarantee that $u^{\theta}_t\in \mathcal{U}$, and obviously, 
$$0\leq J(u^{\theta})-J(u).$$

\begin{lemma}\label{2lm31}
Under the Hypothesis \ref{2H1}, we have,
$$\lim_{\theta\rightarrow0}\mathbb{E}[\sup_{0\leq t \leq T}|X^{\theta}_t-X^u_t|^2]=0.$$
\end{lemma}

\begin{proof}
\indent Note that, for $\tau\in [0,T]$
\begin{eqnarray}
 &&\mathbb{E}\left[\sup_{0\leq t \leq \tau}|X^{\theta}_t-X^u_t|^2\right]\nonumber \\
 &\leq &C\mathbb{E}\int^\tau_0\left|b(t,X^\theta_s,X^\theta_{s-\delta}, P_ {{X^\theta_s}},P_ {{X^\theta_{s-\delta}}},u^\theta_t,u^\theta_{t-\delta})-b(s,\Theta_s)\right|^2ds\nonumber\\
 &&+C\mathbb{E}\int^\tau_0\left|\sigma(t,X^\theta_s,X^\theta_{s-\delta}, P_ {{X^\theta_s}},P_ {{X^\theta_{s-\delta}}},u^\theta_t,u^\theta_{t-\delta})-\sigma(s,\Theta_s)\right|^2ds\nonumber\\
 &\leq &C\mathbb{E}\int^\tau_0\sup_{0\leq s\leq r}\left|X^{\theta}_s-X^u_s\right|^2dr+\theta^2 C\mathbb{E}\int^T_0\left|v_s-u_s\right|^2ds.
 \end{eqnarray}
From Gronwall's inequality we have the desired result.
\end{proof} 

\indent Next, we study the variational process of our state.

\begin{lemma}
Let $K_t$ be the solution of the following linear equation:

\begin{equation}\label{2eq31}
\left\{
\begin{aligned}
  dK_t=&\bigg\{b_x(t,\Theta_t)K_t+\mathbb{E}'\left[b_{\mu }(t,\Theta_t, (X^u_t)')(K_t)'\right]+b_{x_\delta}(t,\Theta_t)K_{t-\delta}&\\
  &+\mathbb{E}'\left[b_{\mu_\delta }(t,\Theta_t, (X^u_{t-\delta})')(K_{t-\delta})'\right]+b_v(t,\Theta_t)(v_t-u_t)&\\
  &+b_{v_\delta}(t,\Theta_t)(v_{t-\delta}-u_{t-\delta})\bigg\}dt&\\
   &+\bigg\{\sigma_x(t,\Theta_t)K_t+\mathbb{E}'\left[\sigma_{\mu }(t,\Theta_t, (X^u_t)')(K_t)'\right]+\sigma_{x_\delta}(t,\Theta_t)K_{t-\delta}&\\
  &+\mathbb{E}'\left[\sigma_{\mu_\delta }(t,\Theta_t, (X^u_{t-\delta})')(K_{t-\delta})'\right]+\sigma_v(t,\Theta_t)(v_t-u_t)&\\
  &+\sigma_{v_\delta}(t,\Theta_t)(v_{t-\delta}-u_{t-\delta})\bigg\}dB_t,\ \ \ \ t\in[0,T],&\\
 K_0=&0,\  v_t=u_t,\  \ \ \ \ \ \ \ \ \ \ \ \ \ \ \ \ \ \ \ \ \ \ \ \ \ \ \ \  \ t\in[-\delta,0]. &
  \end{aligned}
 \right.
\end{equation}  
Then we have 
$$\lim_{\theta\rightarrow0}\mathbb{E}\left[\sup_{0\leq s \leq t}\left| \frac{X^{\theta}_s-X^u_s}{\theta}-K_s\right|^2\right]=0,$$
for all $t\in[0,T].$
\end{lemma}

\begin{proof}
From Theorem \ref{TMFDSDE}, we know equation (\ref{2eq31}) admits a unique solution $K_t$. We set
$$\eta_t=\frac{X^{\theta}_t-X^u_t}{\theta}-K_t, \ \ \ t\in [0,T].$$
Then we have
\begin{eqnarray}
  \eta_t&=&\frac{1}{\theta}\int^t_0\left[b(t,X^\theta_s,X^\theta_{s-\delta}, P_ {{X^\theta_s}},P_ {{X^\theta_{s-\delta}}},u^\theta_t,u^\theta_{t-\delta})-b(s,\Theta_s)\right]ds\nonumber\\
  &&+\frac{1}{\theta}\int^t_0\left[\sigma(t,X^\theta_s,X^\theta_{s-\delta}, P_ {{X^\theta_s}},P_ {{X^\theta_{s-\delta}}},u^\theta_t,u^\theta_{t-\delta})-\sigma(s,\Theta_s)\right]dB_s\nonumber\\
  &&-\int^t_0\bigg\{b_x(s,\Theta_t)K_s+\mathbb{E}'\left[b_{\mu }(s,\Theta_s, (X^u_s)')(K_s)'\right]+b_{x_\delta}(s,\Theta_s)K_{s-\delta}\nonumber\\
  &&+\mathbb{E}'\left[b_{\mu_\delta }(s,\Theta_s, (X^u_{s-\delta})')(K_{s-\delta})'\right]+b_v(s,\Theta_s)(v_s-u_s)\nonumber\\
  &&+b_{v_\delta}(s,\Theta_s)(v_{s-\delta}-u_{s-\delta})\bigg\}ds\nonumber\\
   &&+\bigg\{\sigma_x(s,\Theta_s)K_s+\mathbb{E}'\left[\sigma_{\mu }(s,\Theta_s, (X^u_s)')(K_s)'\right]+\sigma_{x_\delta}(s,\Theta_s)K_{s-\delta}\nonumber\\
  &&+\mathbb{E}'\left[\sigma_{\mu_\delta }(s,\Theta_s, (X^u_{s-\delta})')(K_{s-\delta})'\right]+\sigma_v(s,\Theta_s)(v_s-u_s)\nonumber\\
  &&+\sigma_{v_\delta}(s,\Theta_s)(v_{s-\delta}-u_{s-\delta})\bigg\}dB_s.\nonumber\\
\end{eqnarray}  
 
Since for any $f\in C^{2,1}(\mathcal{P}_2(\mathbb{R}^d))$
\begin{eqnarray}
  &&f(P_{\mu})-f(P_{\mu_0})=\int^1_0\frac{d}{d\lambda}f(P_{\mu_0+\lambda\eta})d\lambda\nonumber\\
 & =&\int^1_0\mathbb{E}\big[f_{\mu}(P_{\mu_0+\lambda\eta},\mu_0+\lambda\eta)\cdot\eta\big]d\lambda,
\end{eqnarray}

 we notice that
\begin{eqnarray*}
  &&\frac{1}{\theta}\int^t_0\Big[b(t,X^\theta_s,X^\theta_{s-\delta}, P_ {{X^\theta_s}},P_ {{X^\theta_{s-\delta}}},u^\theta_s,u^\theta_{s-\delta})\\
  &&-b(s,X^u_s,X^\theta_{s-\delta}, P_ {{X^\theta_s}},P_ {{X^\theta_{s-\delta}}},u^\theta_s,u^\theta_{s-\delta})\Big]ds\\
  &=&\int^t_0 \int^1_0 b_x(s,X^u_s+\lambda\theta(\eta_s+K_s),X^\theta_{s-\delta}, P_ {{X^\theta_s}},P_ {{X^\theta_{s-\delta}}},u^\theta_s,u^\theta_{s-\delta})\\
  &&(\eta_s+K_s)d\lambda ds,\\
  \\
   &&\frac{1}{\theta}\int^t_0\Big[b(t,X^u_s,X^\theta_{s-\delta}, P_ {{X^\theta_s}},P_ {{X^\theta_{s-\delta}}},u^\theta_s,u^\theta_{s-\delta})\\
  &&-b(s,X^u_s,X^u_{s-\delta}, P_ {{X^\theta_s}},P_ {{X^\theta_{s-\delta}}},u^\theta_s,u^\theta_{s-\delta})\Big]ds\\
  &=&\int^t_0 \int^1_0 b_{x_\delta}(s,X^u_s,X^u_{s-\delta}+\lambda\theta(\eta_{s-\delta}+K_{s-\delta}), P_ {{X^\theta_s}},P_ {{X^\theta_{s-\delta}}},u^\theta_s,u^\theta_{s-\delta})\\
  &&(\eta_s+K_s)d\lambda ds,\\
  \\
  &&\frac{1}{\theta}\int^t_0\Big[b(t,X^u_s,X^u_{s-\delta}, P_ {{X^\theta_s}},P_ {{X^\theta_{s-\delta}}},u^\theta_s,u^\theta_{s-\delta})\\
  &&-b(s,X^u_s,X^u_{s-\delta}, P_ {{X^u_s}},P_ {{X^\theta_{s-\delta}}},u^\theta_s,u^\theta_{s-\delta})\Big]ds\\
  &=&\int^t_0 \int^1_0 \mathbb{E}'\Big[b_{\mu }\big(s,X^u_s,X^u_{s-\delta},P_ {X^u_s+\lambda\theta(\eta_s+K_s)},P_ {{X^\theta_{s-\delta}}},u^\theta_s,u^\theta_{s-\delta},\\
  &&(X^u_s+\lambda\theta(\eta_s+K_s))'\big)(\eta_s+K_s)'\Big]d\lambda ds,\\
  \\
    &&\frac{1}{\theta}\int^t_0\Big[b(t,X^u_s,X^u_{s-\delta}, P_ {{X^u_s}},P_ {{X^\theta_{s-\delta}}},u^\theta_s,u^\theta_{s-\delta})\\
  &&-b(s,X^u_s,X^u_{s-\delta}, P_ {{X^u_s}},P_ {{X^u_{s-\delta}}},u^\theta_s,u^\theta_{s-\delta})\Big]ds\\
  &=&\int^t_0 \int^1_0 \mathbb{E}'\Big[b_{\mu_{\delta }}\big(s,X^u_s,X^u_{s-\delta},P_ {{X^u_{s}}},P_ {X^u_{s-\delta}+\lambda\theta(\eta_{s-\delta}+K_{s-\delta})},\\
  &&u^\theta_s,u^\theta_{s-\delta},(X^u_{s-\delta}+\lambda\theta(\eta_{s-\delta}+K_{s-\delta}))'\big)(\eta_s+K_s)'\Big]d\lambda ds,\\
  \\
     &&\frac{1}{\theta}\int^t_0\Big[b(t,X^u_s,X^u_{s-\delta}, P_ {{X^u_s}},P_ {{X^u_{s-\delta}}},u^\theta_s,u^\theta_{s-\delta})\\
  &&-b(s,X^u_s,X^u_{s-\delta}, P_ {{X^u_s}},P_ {{X^u_{s-\delta}}},u_s,u^\theta_{s-\delta})\Big]ds\\
  &=&\int^t_0 \int^1_0 b_v(s,X^u_s,X^u_{s-\delta}, P_ {{X^u_s}},P_ {{X^u_{s-\delta}}},u_s+\lambda\theta(v_s-u_s),u^\theta_{s-\delta})\\
  &&(v_s-u_s)d\lambda ds,\\
  \\
     &&\frac{1}{\theta}\int^t_0\Big[b(t,X^u_s,X^u_{s-\delta}, P_ {{X^u_s}},P_ {{X^u_{s-\delta}}},u_s,u^\theta_{s-\delta})\\
  &&-b(s,X^u_s,X^u_{s-\delta}, P_ {{X^u_s}},P_ {{X^u_{s-\delta}}},u_s,u_{s-\delta})\Big]ds\\
  &=&\int^t_0 \int^1_0 b_{v_\delta}(s,X^u_s,X^u_{s-\delta}, P_ {{X^u_s}},P_ {{X^u_{s-\delta}}},u_s,u^\theta_{s-\delta}+\lambda\theta(v_{s-\delta}-u_{s-\delta}))\\
  &&(v_{s-\delta}-u_{s-\delta})d\lambda ds.\\
\end{eqnarray*}  

Similarly result can be obtained for $\sigma$. On the other hand, we have
\begin{eqnarray*} 
&&\frac{1}{\theta}\int^t_0\Big[b(t,X^u_s,X^u_{s-\delta}, P_ {{X^\theta_s}},P_ {{X^\theta_{s-\delta}}},u^\theta_s,u^\theta_{s-\delta})\\
  &&-b(s,X^u_s,X^u_{s-\delta}, P_ {{X^u_s}},P_ {{X^\theta_{s-\delta}}},u^\theta_s,u^\theta_{s-\delta})\Big]ds\\
   &&-\int^t_0\mathbb{E}'\left[b_{\mu }(s,\Theta_s, (X^u_s)')(K_s)'\right]ds\\
&=&\int^t_0 \int^1_0 \mathbb{E}'\Big[b_{\mu }\big(s,X^u_s,X^u_{s-\delta},P_ {X^u_s+\lambda\theta(\eta_s+K_s)},P_ {{X^\theta_{s-\delta}}},u^\theta_s,u^\theta_{s-\delta},\\
&&(X^u_s+\lambda\theta(\eta_s+K_s))'\big)(\eta_s)'\Big]d\lambda ds\\
  &&+\int^t_0 \int^1_0 \mathbb{E}'\bigg\{\Big[b_{\mu }\big(s,X^u_s,X^u_{s-\delta},P_ {X^u_s+\lambda\theta(\eta_s+K_s)},P_ {{X^\theta_{s-\delta}}},u^\theta_s,u^\theta_{s-\delta},\\
  &&(X^u_s+\lambda\theta(\eta_s+K_s))')-b_{\mu }(s,\Theta_s, (X^u_s)'\big)\Big](K_s)'\bigg\}d\lambda ds.\\
\end{eqnarray*} 
Set 
 \begin{eqnarray*}
 I^{\theta}_t&=&\int^t_0 \int^1_0 \mathbb{E}'\bigg\{\Big[b_{\mu }\big(s,X^u_s,X^u_{s-\delta},P_ {X^u_s+\lambda\theta(\eta_s+K_s)},P_ {{X^\theta_{s-\delta}}},u^\theta_s,u^\theta_{s-\delta},\\
 &&(X^u_s+\lambda\theta(\eta_s+K_s))'\big)-b_{\mu }\left(s,\Theta_s, (X^u_s)'\right)\Big](K_s)'\bigg\}d\lambda ds.\\
\end{eqnarray*}  
Then, from Lemma \ref{2lm31}, Lipschitz continuity and the definition of 2-Wasserstein metric, we have
$$\lim_{\theta\rightarrow0}\mathbb{E}\big[\sup_{0\leq s \leq T}|I^{\theta}_s|^2\big]=0.$$

Therefore, we have
\begin{eqnarray}
 &&\mathbb{E}\big[\sup_{0\leq s \leq t}|{\eta}_s|^2\big]\nonumber\\
&\leq &C \mathbb{E}\int^t_0 \int^1_0 \Big|b_x(s,X^u_s+\lambda\theta(\eta_s+K_s),X^\theta_{s-\delta}, P_ {{X^\theta_s}},P_ {{X^\theta_{s-\delta}}},u^\theta_s,u^\theta_{s-\delta})\nonumber\\
  &&(\eta_s)\Big|^2 d\lambda ds\nonumber\\
 &&+C \mathbb{E}\int^t_0 \int^1_0 \Big| b_{x_\delta}(s,X^u_s,X^u_{s-\delta}+\lambda\theta(\eta_{s-\delta}+K_{s-\delta}), P_ {{X^\theta_s}},P_ {{X^\theta_{s-\delta}}},u^\theta_s,u^\theta_{s-\delta})\nonumber\\
  &&(\eta_s)\Big|^2d\lambda ds\nonumber\\
 &&+C \mathbb{E}\int^t_0 \int^1_0 \mathbb{E}'\Big[b_{\mu }\big(s,X^u_s,X^u_{s-\delta},P_ {X^u_s+\lambda\theta(\eta_s+K_s)},P_ {{X^\theta_{s-\delta}}},u^\theta_s,u^\theta_{s-\delta},\\
 &&(X^u_s+\lambda\theta(\eta_s+K_s))'\big)(\eta_s)'\Big]^2d\lambda ds\nonumber\\
 &&+C \mathbb{E}\int^t_0 \int^1_0 \mathbb{E}'\Big[b_{\mu_{\delta }}\big(s,X^u_s,X^u_{s-\delta},P_ {{X^u_{s}}},P_ {X^u_{s-\delta}+\lambda\theta(\eta_{s-\delta}+K_{s-\delta})},\nonumber\\
  &&u^\theta_s,u^\theta_{s-\delta},(X^u_{s-\delta}+\lambda\theta(\eta_{s-\delta}+K_{s-\delta}))'\big)(\eta_s)'\Big]^2d\lambda ds\nonumber\\
 &&+C \mathbb{E}\int^t_0 \int^1_0 \Big|\sigma_x(s,X^u_s+\lambda\theta(\eta_s+K_s),X^\theta_{s-\delta}, P_ {{X^\theta_s}},P_ {{X^\theta_{s-\delta}}},u^\theta_s,u^\theta_{s-\delta})\nonumber\\
  &&(\eta_s)\Big|^2 d\lambda ds\nonumber\\
 &&+C \mathbb{E}\int^t_0 \int^1_0 \Big| \sigma_{x_\delta}(s,X^u_s,X^u_{s-\delta}+\lambda\theta(\eta_{s-\delta}+K_{s-\delta}), P_ {{X^\theta_s}},P_ {{X^\theta_{s-\delta}}},u^\theta_s,u^\theta_{s-\delta})\nonumber\\
  &&(\eta_s)\Big|^2d\lambda ds\nonumber\\
 &&+C \mathbb{E}\int^t_0 \int^1_0 \mathbb{E}'\Big[\sigma_{\mu }\big(s,X^u_s,X^u_{s-\delta},P_ {X^u_s+\lambda\theta(\eta_s+K_s)},P_ {{X^\theta_{s-\delta}}},u^\theta_s,u^\theta_{s-\delta},\nonumber\\
 &&(X^u_s+\lambda\theta(\eta_s+K_s))'\big)(\eta_s)'\Big]^2d\lambda ds\nonumber\\
 &&+C \mathbb{E}\int^t_0 \int^1_0 \mathbb{E}'\Big[\sigma_{\mu_{\delta }}\big(s,X^u_s,X^u_{s-\delta},P_ {{X^u_{s}}},P_ {X^u_{s-\delta}+\lambda\theta(\eta_{s-\delta}+K_{s-\delta})},\nonumber\\
  &&u^\theta_s,u^\theta_{s-\delta},(X^u_{s-\delta}+\lambda\theta(\eta_{s-\delta}+K_{s-\delta}))'\big)(\eta_s)'\Big]^2d\lambda ds\nonumber\\
 &&+C\mathbb{E}\big[\sup_{0\leq s \leq t}|\beta^{\theta}_s|^2\big]
\end{eqnarray}  
where
\begin{eqnarray}
  \beta^{\theta}_t&=&\int^t_0 \int^1_0\Big[b_x(s,X^u_s+\lambda\theta(\eta_s+K_s),X^\theta_{s-\delta}, P_ {{X^\theta_s}},P_ {{X^\theta_{s-\delta}}},u^\theta_s,u^\theta_{s-\delta})\nonumber\\
  &&-b_x(t,\Theta_t)\Big](K_s)d\lambda ds\nonumber\\
  &&+\int^t_0 \int^1_0 \Big[ b_{x_\delta}(s,X^u_s,X^u_{s-\delta}+\lambda\theta(\eta_{s-\delta}+K_{s-\delta}), P_ {{X^\theta_s}},P_ {{X^\theta_{s-\delta}}},u^\theta_s,u^\theta_{s-\delta})\nonumber\\
  &&-b_{x_\delta}(t,\Theta_t)\Big](K_s)d\lambda ds\nonumber\\
  &&+\int^t_0 \int^1_0 \mathbb{E}'\bigg\{\Big[b_{\mu }\big(s,X^u_s,X^u_{s-\delta},P_ {X^u_s+\lambda\theta(\eta_s+K_s)},P_ {{X^\theta_{s-\delta}}},u^\theta_s,u^\theta_{s-\delta},\nonumber\\
  &&(X^u_s+\lambda\theta(\eta_s+K_s))'\big)-b_{\mu }(t,\Theta_t, (X^u_t)')\Big](K_s)'\bigg\}d\lambda ds\nonumber\\
  &&+\int^t_0 \int^1_0 \mathbb{E}'\bigg\{\Big[b_{\mu_{\delta }}\big(s,X^u_s,X^u_{s-\delta},P_ {{X^u_{s}}},P_ {X^u_{s-\delta}+\lambda\theta(\eta_{s-\delta}+K_{s-\delta})},\nonumber\\
  &&u^\theta_s,u^\theta_{s-\delta},(X^u_{s-\delta}+\lambda\theta(\eta_{s-\delta}+K_{s-\delta}))'\big)-b_{\mu_\delta }(t,\Theta_t, (X^u_{t-\delta})')\Big](K_s)'\bigg\}d\lambda ds\nonumber\\
 &&+\int^t_0 \int^1_0 b_v(s,X^u_s,X^u_{s-\delta}, P_ {{X^u_s}},P_ {{X^u_{s-\delta}}},u_s+\lambda\theta(v_s-u_s),u^\theta_{s-\delta})\nonumber\\
  &&(v_s-u_s)d\lambda ds\nonumber\\
  &&+\int^t_0 \int^1_0 b_{v_\delta}(s,X^u_s,X^u_{s-\delta}, P_ {{X^u_s}},P_ {{X^u_{s-\delta}}},u_s,u^\theta_{s-\delta}+\lambda\theta(v_{s-\delta}-u_{s-\delta}))\nonumber\\
  &&(v_{s-\delta}-u_{s-\delta})d\lambda ds\nonumber\\
  &&+\int^t_0 \int^1_0\Big[\sigma_x(s,X^u_s+\lambda\theta(\eta_s+K_s),X^\theta_{s-\delta}, P_ {{X^\theta_s}},P_ {{X^\theta_{s-\delta}}},u^\theta_s,u^\theta_{s-\delta})\nonumber\\
  &&-\sigma_x(t,\Theta_t)\Big](K_s)d\lambda ds\nonumber\\
  &&+\int^t_0 \int^1_0 \Big[ \sigma_{x_\delta}(s,X^u_s,X^u_{s-\delta}+\lambda\theta(\eta_{s-\delta}+K_{s-\delta}), P_ {{X^\theta_s}},P_ {{X^\theta_{s-\delta}}},u^\theta_s,u^\theta_{s-\delta})\nonumber\\
  &&-\sigma_{x_\delta}(t,\Theta_t)\Big](K_s)d\lambda ds\nonumber\\
  &&+\int^t_0 \int^1_0 \mathbb{E}'\bigg\{\Big[\sigma_{\mu }\big(s,X^u_s,X^u_{s-\delta},P_ {X^u_s+\lambda\theta(\eta_s+K_s)},P_ {{X^\theta_{s-\delta}}},u^\theta_s,u^\theta_{s-\delta},\nonumber\\
  &&(X^u_s+\lambda\theta(\eta_s+K_s))'\big)-\sigma_{\mu }(t,\Theta_t, (X^u_t)')\Big](K_s)'\bigg\}d\lambda ds\nonumber\\
  &&+\int^t_0 \int^1_0 \mathbb{E}'\bigg\{\Big[\sigma_{\mu_{\delta }}\big(s,X^u_s,X^u_{s-\delta},P_ {{X^u_{s}}},P_ {X^u_{s-\delta}+\lambda\theta(\eta_{s-\delta}+K_{s-\delta})},\nonumber\\
  &&u^\theta_s,u^\theta_{s-\delta},(X^u_{s-\delta}+\lambda\theta(\eta_{s-\delta}+K_{s-\delta}))'\big)-\sigma_{\mu_\delta }(t,\Theta_t, (X^u_{t-\delta})')\Big](K_s)'\bigg\}d\lambda ds\nonumber\\
 &&+\int^t_0 \int^1_0 \sigma_v(s,X^u_s,X^u_{s-\delta}, P_ {{X^u_s}},P_ {{X^u_{s-\delta}}},u_s+\lambda\theta(v_s-u_s),u^\theta_{s-\delta})\nonumber\\
  &&(v_s-u_s)d\lambda ds\nonumber\\
  &&+\int^t_0 \int^1_0 \sigma_{v_\delta}(s,X^u_s,X^u_{s-\delta}, P_ {{X^u_s}},P_ {{X^u_{s-\delta}}},u_s,u^\theta_{s-\delta}+\lambda\theta(v_{s-\delta}-u_{s-\delta}))\nonumber\\
  &&(v_{s-\delta}-u_{s-\delta})d\lambda ds\nonumber\\
\end{eqnarray}
Proceeding as in the estimate of $I^{\theta}_t$, we can prove that
$$\lim_{\theta\rightarrow0}\mathbb{E}\big[\sup_{0\leq s \leq T}|\beta^{\theta}_s|^2\big]=0.$$ 
Since the derivatives of $b,\sigma$ are bounded, we deduce that
$$\mathbb{E}\big[\sup_{0\leq s \leq t}|\eta_s|^2\big]\leq C\mathbb{E}\int^t_0|\eta_s|^2ds+C\mathbb{E}\big[\sup_{0\leq s \leq t}|\beta^{\theta}_s|^2\big].$$
Finally, by Gronwall's inequality, we complete the proof.
\end{proof}

\begin{lemma}
Let $u$ be an optimal control and $X^u_t$ be the corresponding optimal trajectory. Then, for any control $v\in \mathcal{U}$, we get
 \begin{eqnarray}
  0&\leq &\mathbb{E}\big\{\Phi_x(X^u_T,P_{X^u_t})(K_T)+\mathbb{E}'[\Phi_{\mu}(X^u_T,P_{X^u_t},(X^u_T)')(K_T)']\big\}\nonumber\\
  &&+\mathbb{E}\int^T_0\bigg\{ h_x(t,X^u_t,P_{X^u_t},u_t,u_{t-\delta})(K_t)+\mathbb{E}'\big[h_{\mu}\big(t,X^u_t,P_{X^u_t},u_t,u_{t-\delta},\nonumber\\
  &&(X^u_T)'\big)(K_t)'\big]
  +h_v(t,X^u_t,P_{X^u_t},u_t,u_{t-\delta})(v_t-u_t)\nonumber\\
  &&+h_{v_\delta}(t,X^u_t,P_{X^u_t},u_t,u_{t-\delta})(v_{t-\delta}-u_{t-\delta})\bigg\} dt
\end{eqnarray} 
\end{lemma}

\begin{proof}
Since $u$ is an optimal control, we deduce
\begin{eqnarray}\label{2eq32}
  0&\leq &J(u^{\theta}_t)-J(u_t)\nonumber\\
  &=&\mathbb{E}\big[\Phi(X^{\theta}_T,P_{X^{\theta}_T})-\Phi(X^u_T,P_{X^u_T})\big]\nonumber\\
  &&+\mathbb{E}\int^T_0\big[h(t,X^{\theta}_t,P_{X^{\theta}_t},u^{\theta}_t,u^{\theta}_{t-\delta})-h(t,X^u_t,P_{X^u_t},u^{\theta}_t,u^{\theta}_{t-\delta})\big]dt\nonumber\\
  &&+\mathbb{E}\int^T_0\big[h(t,X^u_t,P_{X^u_t},u^{\theta}_t,u^{\theta}_{t-\delta})-h(t,X^u_t,P_{X^u_t},u_t,{u_{t-\delta}})\big]dt\nonumber\\
 & =:&I_1+I_2+I_3.
\end{eqnarray}

\begin{eqnarray}
I_1&=&\mathbb{E}\left[\int^1_0\Phi_x(X^u_T+\lambda\theta(\eta_T+K_T),P_{X^{\theta}_T})\theta(\eta_{T}+K_T)d\lambda\right]\nonumber\\
&&+\mathbb{E}\bigg\{\int^1_0 \mathbb{E}'\Big[ \Phi_{\mu}\left(X^u_T,P_{X^u_T+\lambda\theta(\eta_T+K_T)},(X^u_T+\lambda\theta(\eta_T+K_T))'\right)\nonumber\\
&&\theta(\eta_{T}+K_T)'\Big]d\lambda\bigg\},\nonumber\\
 I_2&=&\mathbb{E}\left[\int^T_0\int^1_0 h_x(t,X^u_t+\lambda\theta(\eta_t+K_t),P_{X^{\theta}_t},u^{\theta}_t,u^{\theta}_{t-\delta})\theta(\eta_{t}+K_t)d\lambda dt\right]\nonumber\\
&&+\mathbb{E}\bigg\{ \int^T_0 \int^1_0 \mathbb{E}'\Big[ h_{\mu }\Big(t,X^u_t,P_{X^u_t+\lambda\theta(\eta_t+K_t)},u^{\theta}_t,u^{\theta}_{t-\delta},\nonumber\\
&&(X^u_t+\lambda\theta(\eta_t+K_t))'\Big)\theta(\eta_{t}+K_t)'\Big]d\lambda dt \bigg\},\nonumber\\
I_3&=&\mathbb{E}\bigg[\int^T_0\int^1_0 h_v\left(t,X^u_t,P_{X^u_t},u_t+\lambda\theta(v_t-u_t),u^{\theta}_{t-\delta}\right)\nonumber\\
&&\theta(v_{t}-u_t)d\lambda dt\bigg]\nonumber\\
&&+\mathbb{E}\bigg[\int^T_0\int^1_0 h_{v_\delta}\left(t,X^u_t,P_{X^u_t},u_t,u_{t-\delta}+\lambda\theta(v_{t-\delta}-u_{t-\delta})\right)\nonumber\\
&&\theta(v_{t-\delta}-u_{t-\delta})d\lambda dt\bigg].
\end{eqnarray}
From (\ref{2eq32}), we get
\begin{eqnarray}
 0&\leq &\mathbb{E}\left[\int^1_0\Phi_x(X^u_T+\lambda\theta(\eta_T+K_T),P_{X^{\theta}_T})(K_T)d\lambda\right]\nonumber\\
&&+\mathbb{E}\bigg\{\int^1_0 \mathbb{E}'\Big[ \Phi_{\mu}\left(X^u_T,P_{X^u_T+\lambda\theta(\eta_T+K_T)},(X^u_T+\lambda\theta(\eta_T+K_T))'\right)\nonumber\\
&&(K_T)'\Big]d\lambda\bigg\},\nonumber\\
&&+\mathbb{E}\left[\int^T_0\int^1_0 h_x(t,X^u_t+\lambda\theta(\eta_t+K_t),P_{X^{\theta}_t},u^{\theta}_t,u^{\theta}_{t-\delta})(K_t)d\lambda dt\right]\nonumber\\
&&+\mathbb{E}\bigg\{ \int^T_0 \int^1_0 \mathbb{E}'\Big[ h_{\mu }\Big(t,X^u_t,P_{X^u_t+\lambda\theta(\eta_t+K_t)},u^{\theta}_t,u^{\theta}_{t-\delta},\nonumber\\
&&(X^u_t+\lambda\theta(\eta_t+K_t))'\Big)(K_t)'\Big]d\lambda dt \bigg\}\nonumber\\
&&+\mathbb{E}\bigg[\int^T_0\int^1_0 h_v\left(t,X^u_t,P_{X^u_t},u_t+\lambda\theta(v_t-u_t),u^{\theta}_{t-\delta}\right)\nonumber\\
&&(v_{t}-u_t)d\lambda dt\bigg]\nonumber\\
&&+\mathbb{E}\bigg[\int^T_0\int^1_0 h_{v_\delta}\left(t,X^u_t,P_{X^u_t},u_t,u_{t-\delta}+\lambda\theta(v_{t-\delta}-u_{t-\delta})\right)\nonumber\\
&&(v_{t-\delta}-u_{t-\delta})d\lambda dt\bigg]\nonumber\\
&&+\rho^{\theta}_t.
\end{eqnarray}

where

\begin{eqnarray*}
\rho^{\theta}_t&= &\mathbb{E}\left[\int^1_0\Phi_x(X^u_T+\lambda\theta(\eta_T+K_T),P_{X^{\theta}_T})(\eta_T)d\lambda\right]\nonumber\\
&&+\mathbb{E}\bigg\{\int^1_0 \mathbb{E}'\Big[ \Phi_{\mu}\left(X^u_T,P_{X^u_T+\lambda\theta(\eta_T+K_T)},(X^u_T+\lambda\theta(\eta_T+K_T))'\right)\nonumber\\
&&(\eta_T)'\Big]d\lambda\bigg\},\nonumber\\
&&+\mathbb{E}\left[\int^T_0\int^1_0 h_x(t,X^u_t+\lambda\theta(\eta_t+K_t),P_{X^{\theta}_t},u^{\theta}_t,u^{\theta}_{t-\delta})(\eta_t)d\lambda dt\right]\nonumber\\
&&+\mathbb{E}\bigg\{ \int^T_0 \int^1_0 \mathbb{E}'\Big[ h_{\mu }\Big(t,X^u_t,P_{X^u_t+\lambda\theta(\eta_t+K_t)},u^{\theta}_t,u^{\theta}_{t-\delta},\nonumber\\
&&(X^u_t+\lambda\theta(\eta_t+K_t))'\Big)(\eta_t)'\Big]d\lambda dt \bigg\},\nonumber\\
\end{eqnarray*}
Since the derivatives of $\Phi,h$ are bounded and
 $$\lim_{\theta\rightarrow0}\mathbb{E}\big[\sup_{0\leq s \leq T}|\eta_s|^2\big]=0.$$
 we get
  $$\lim_{\theta\rightarrow0}\rho^{\theta}_t=0.$$

From the fact $u^{\theta}_t\rightarrow u_t$ and Lipschitz continuity of $\Phi, h$ we obtain the result.

\end{proof}

\section{Necessary and sufficient conditions for the optimal control}

\indent It is well known that there is a duality relationship between the stochastic delay differential equations and the backward anticipated stochastic differential equations. In this section, we introduce the adjoint process with the help of mean-filed backward anticipated stochastic differential equation, then the variational inequality can be deduced.\\
\indent Let us consider the following adjoint equation:
\begin{equation}\label{2eq41}
\left\{
\begin{aligned}
-dp_t=& \bigg\{ b^*_x(t,\Theta_t)p_t+\sigma^*_x(t,\Theta_t)q_t+h_x(t,X^u_t,P_{X^u_t},u_t,u_{t-\delta})&\\
&+\mathbb{E}'\Big[b^*_{\mu }(t,\Theta'_t,X^u_t)(p_t)'+\sigma^*_{\mu }(t,\Theta'_t,X^u_t)(q_t)'&\\
&+h_{\mu }(t,(X^u_t)',P_{X^u_t},(u_t)',(u_{t-\delta})',X^u_t)\Big]&\\
&+\mathbb{E}^{\mathcal{F}_t}\Big[b^*_{x_\delta}(t,\Theta_t)|_{t+\delta} p_{t+\delta}\Big]+\mathbb{E}^{\mathcal{F}_t}\Big[\sigma^*_{x_\delta}(t,\Theta_t)|_{t+\delta} q_{t+\delta}\Big]&\\
&+\mathbb{E}'\bigg[\mathbb{E}^{\mathcal{F}_t}\Big[b^*_{\mu_\delta}(t,\Theta'_t,X^u_t)|_{t+\delta} (p_{t+\delta})'\Big]&\\
&+\mathbb{E}^{\mathcal{F}_t}\Big[\sigma^*_{\mu_\delta}(t,\Theta'_t,X^u_t)|_{t+\delta} (q_{t+\delta})'\Big]\bigg]\bigg \} dt-q_tdB_t,&\\
p_T=&\Phi_x(X^u_t,P_{X^u_t})+\mathbb{E}'\big[\Phi_{\mu }((X^u_t)',(P_{X^u_t})',X^u_t)\big],&\\
p_t=&0, q_t=0, \ \ \ \ \ \ \ \ t\in(T,T+\delta],&
\end{aligned}
\right.
\end{equation}
where $b^*$ denotes the transpose of $b$; $b_{x_\delta}(t,\Theta_t)|_{t+\delta}$ denotes the value of $b_{x_\delta}(t,\Theta_t)$ when $t$ replaced by $t+\delta$, other involved terms are defined similarly.\\
\indent From Theorem \ref{unique}, (\ref{2eq41}) admits a unique adapted solution. Then, we get the main result of this paper which is stochastic maximum principle for generalized mean-field delay control problem.

\begin{theorem}
(Necessary Conditions for the Optimal Control). Let $u$ be an optimal control, and $X^u_t$ denote the associated optimal trajectory. Let $(p,q)$ be the unique solution of equation (\ref{2eq41}). Then the following integral stochastic maximum principle holds: for all $v\in \mathcal{U},$
\begin{eqnarray}\label{2eq43}
\left<H_v(t,\Theta_t,p_t,q_t),v-u_t\right>+\left<\mathbb{E}^{\mathcal{F}_t}\left[H_{v_\delta}(t,\Theta_t,p_t,q_t)|_{t+\delta}\right],v-u_t\right>\geq 0,
\end{eqnarray}
where,
\begin{eqnarray*}
H(t,\Theta,p,q)&=&b^*(t,x,x_\delta,\mu,\mu_\delta,v,v_\delta)p+\sigma^* (t,x,x_\delta,\mu,\mu_\delta,v,v_\delta)q\\
&&+h(t,x,\mu,v,v_\delta)
\end{eqnarray*} 
\end{theorem}

 \begin{proof}
 By applying It\^o's formula to $\left<p_t,K_t\right>$, notice that
 \begin{eqnarray*}
&&\mathbb{E}\int^T_0\bigg\{K^*_{t-\delta}b^*_{x_\delta}(t,\Theta_t)p_t-K^*_t\mathbb{E}^{\mathcal{F}_t}\left[b^*_{x_\delta}(t,\Theta_t)|_{t+\delta} p_{t+\delta}\right]\bigg\}dt \\
&=&\mathbb{E}\bigg\{\int^T_0K^*_{t-\delta}b^*_{x_\delta}(t,\Theta_t)p_tdt-\int_\delta^{T+\delta} K^*_{t-\delta}b^*_{x_\delta}(t,\Theta_t) p_{t}dt\bigg\} \\
&=&\mathbb{E}\bigg\{\int^\delta_0K^*_{t-\delta}b^*_{x_\delta}(t,\Theta_t)p_tdt-\int_T^{T+\delta} K^*_{t-\delta}b^*_{x_\delta}(t,\Theta_t) p_{t}dt\bigg\} \\
&=&0,
\end{eqnarray*}
  and similarly results can be obtained for other terms, then we get
\begin{eqnarray}\label{2eq42}
&&\mathbb{E}\int^T_0\Big\{\left<H_v(t,\Theta_t,p_t,q_t),v_t-u_t\right>+\left<H_{v_\delta}(t,\Theta_t,p_t,q_t),v_{t-\delta}-u_{t-\delta}\right>\Big\}dt\nonumber\\
&=&\mathbb{E}\left<\Phi_x(X^u_t,P_{X^u_t})+\mathbb{E}'\big[\Phi_{\mu }((X^u_t)',(P_{X^u_t})',X^u_t)\big],K_T\right>\nonumber \\
&&+\mathbb{E}\int^T_0 \left<h_x(t,X^u_t,P_{X^u_t},u_t,u_{t-\delta}),K_t\right>dt\nonumber \\
&&+\mathbb{E}\int^T_0 \left<\mathbb{E}'\big[ h_{\mu }(t,(X^u_t)',P_{X^u_t}, (u_t)',(u_{t-\delta})',X^u_t)\big],K_t\right>dt\nonumber  \\
&&+\mathbb{E}\int^T_0 \left< h_{v }(t,(X^u_t)',P_{X^u_t}, X^u_t, (u_t)',(u_{t-\delta})'),v_t-u_t\right>dt\nonumber  \\
&&+\mathbb{E}\int^T_0 \left< h_{v_\delta }(t,(X^u_t)',P_{X^u_t}, X^u_t, (u_t)',(u_{t-\delta})'),v_{t-\delta}-u_{t-\delta}\right>dt\nonumber  \\
&\geq&0.
\end{eqnarray}
\indent Set
\begin{eqnarray*}
v_s=\left\{
\begin{array}{l}
v_s, \ \ \ s\in[t,t+\varepsilon),\\
u_s, \ \ \ otherwise,
\end{array}
\right.
\end{eqnarray*}
where $t\in[0,T]$, $v\in \mathcal{U}$. Then (\ref{2eq42}) leads to
\begin{eqnarray}
&&\frac{1}{\varepsilon}\mathbb{E}\int^{t+\varepsilon}_t\left<H_v(s,\Theta_s,p_s,q_s),v_s-u_s\right>ds\nonumber\\
&&+\frac{1}{\varepsilon}\mathbb{E}\int^{t+\varepsilon+\delta}_{t+\delta}\left<H_{v_\delta}(s,\Theta_s,p_s,q_s),v_{s-\delta}-u_{s-\delta}\right>ds\geq0.
\end{eqnarray}
That means
\begin{eqnarray}
&&\frac{1}{\varepsilon}\mathbb{E}\int^{t+\varepsilon}_t\left<H_v(s,\Theta_s,p_s,q_s)+\mathbb{E}^{\mathcal{F}_s}\left[H_{v_\delta}(s,\Theta_s,p_s,q_s)|_{s+\delta}\right],v_s-u_s\right>ds\nonumber\\
&\geq& 0
\end{eqnarray}
Letting $\varepsilon\rightarrow 0+$, by Lebesgue differential theorem, we have
\begin{eqnarray}
&&\mathbb{E}\left<H_v(t,\Theta_t,p_t,q_t)+\mathbb{E}^{\mathcal{F}_t}\left[H_{v_\delta}(t,\Theta_t,p_t,q_t)|_{t+\delta}\right],v_t-u_t\right>\geq 0 \ \ a.e.\nonumber.
\end{eqnarray}
\indent Now, let $v\in \mathcal{U}$ be a selected element and $A$ an arbitrary element of $\sigma$-algebra $\mathcal{F}_t$, set $w_t=v\textbf{1}_A+u_t\textbf{1}_{A^c}$. Clearly, $w_t$ is an admissible control and for all $A\in \mathcal{F}_t,$ we obtain
\begin{eqnarray}
&&\mathbb{E}\left<H_v(t,\Theta_t,p_t,q_t)+\mathbb{E}^{\mathcal{F}_t}\left[H_{v_\delta}(t,\Theta_t,p_t,q_t)|_{t+\delta}\right],w_t-u_t\right>\nonumber\\
&=&\mathbb{E}\left[\left<H_v(t,\Theta_t,p_t,q_t)+\mathbb{E}^{\mathcal{F}_t}\left[H_{v_\delta}(t,\Theta_t,p_t,q_t)|_{t+\delta}\right],v-u_t\right>\textbf{1}_A\right]\nonumber\\
&\geq& 0 \ \ a.e.\nonumber,
\end{eqnarray}
which implies
\begin{eqnarray}
&&\mathbb{E}\left[\left<H_v(t,\Theta_t,p_t,q_t)+\mathbb{E}^{\mathcal{F}_t}\left[H_{v_\delta}(t,\Theta_t,p_t,q_t)|_{t+\delta}\right],v-u_t\right>|\mathcal{F}_t\right]\nonumber\\
&=&\left<H_v(t,\Theta_t,p_t,q_t)+\mathbb{E}^{\mathcal{F}_t}\left[H_{v_\delta}(t,\Theta_t,p_t,q_t)|_{t+\delta}\right],v-u_t\right>\geq 0 \ \ a.e.\nonumber.
\end{eqnarray}
 \end{proof}

\indent Then we study the sufficient conditions.
\begin{theorem}
(Sufficient Conditions for the Optimality of Control). Let Hypothesis \ref{2H1} hold and let $u$ is the control satisfies (\ref{2eq43}) and $(p,q)$ be the unique solution of (\ref{2eq41}). We further assume $\Phi(x,\mu ), H(t,x,x_\delta,\mu,\mu_\delta, p_t,q_t,v,v_\delta)$ are convex respect to $(x,\mu)$ and $(x,x_\delta,\mu,\mu_\delta,v,v_\delta)$. Then $u$ is the optimal control of our control problem.
\end{theorem}

\begin{proof}
For any $v\in \mathcal{U}$, we have
\begin{eqnarray}
&&J(v)-J(u)\nonumber \\
&=&\mathbb{E}\big[\Phi(X^v_T,P_{X^v_T})-\Phi(X^u_T,P_{X^u_T})\big]\nonumber\\
&&+\mathbb{E}\int^T_0 \big[h(t,X^v_t,P_{X^v_t},
v_t,v_{t-\delta} )-h(t,{X}^{u}_t,P_{X^u_t},
u_t ,u_{t-\delta})\big]dt.
\end{eqnarray}
\indent Since $\Phi$ is convex with respect to $x$. we get
\begin{eqnarray}
&&\Phi(X^{v}_T,P_{X^v_T})-\Phi(X^{u}_T,P_{X^u_T})\nonumber \\
&\geq &\Phi_x(X^{u}_T,P_{X^u_T})(X^{v}_T-X^u_T)\nonumber\\
&&+\mathbb{E}'\big[\Phi_\mu(X^{u}_T,P_{X^u_T},(X^{u}_T)')(X^{v}_T-X^u_T)'\big]. 
\end{eqnarray}
Consequently
\begin{eqnarray}
&&J(v)-J(u)\nonumber\\
&=&\mathbb{E}\big\{\Phi_x(X^{u}_T,P_{X^u_T})(X^{v}_T-X^u_T)\nonumber\\
&&+\mathbb{E}'\big[\Phi_x(X^{u}_T,P_{X^u_T},(X^{u}_T)')(X^{v}_T-X^u_T)'\big]\big\}\nonumber \\
&&+\mathbb{E}\int^T_0 \big[h(t,X^v_t,P_{X^v_t},
v_t,v_{t-\delta} )-h(t,{X}^{u}_t,P_{X^u_t},
u_t ,u_{t-\delta})\big]dt.
\end{eqnarray}
\indent By applying It\^{o}'s formula to $\left<p_t,X^{v}_t-{X}^u_t\right>$ and taking the expectation, we obtain

\begin{eqnarray}
&&J(v)-J(u)\\
&\geq &\mathbb{E}\int^T_0\big [H(t,X^{v}_t,X^{v}_{t-\delta},P_{X^v_t},P_{X^v_{t-\delta}},
v_t,v_{t-\delta},p_t,q_t )-H(t,\Theta_t,p_t,q_t )\big]dt\nonumber \\
&&-\mathbb{E}\int^T_0 \bigg\{\left<H_x(t,\Theta_t,p_t,q_t),X^{v}_t-{X}^u_t\right>\nonumber \\
&&+\mathbb{E}'\left[\left< H_\mu(t,\Theta_t,({X}^{u}_t)',p_t,q_t,
),(X^{v}_t-{X}^u_t)'\right>\right]\bigg\}dt\nonumber\\
&&-\mathbb{E}\int^T_0 \bigg\{\left<\mathbb{E}^{\mathcal{F}_t}\left[H_{x_\delta}(t,\Theta_t,p_t,q_t)|_{t+\delta}\right],X^{v}_t-{X}^u_t\right>\nonumber \\
&&+\mathbb{E}'\left[\left<\mathbb{E}^{\mathcal{F}_t}\left[ H_{\mu_\delta}(t,\Theta_t,({X}^{u}_t)',p_t,q_t,
)|_{t+\delta}\right],(X^{v}_t-{X}^u_t)'\right>\right]\bigg\}dt
\end{eqnarray}
\indent Since $H$ is convex with respect to $(x,x_\delta,\mu,\mu_\delta,v,v_\delta)$  The use of Clark generalized gradient of $H$, evaluated at $({X}^u_t,{X}^u_{t-\delta},P_{{X}^u_t},P_{{X}^u_{t-\delta}},u_t,u_{t-\delta})$, yields

\begin{eqnarray}
&&H(t,X^{v}_t,X^{v}_{t-\delta},P_{X^v_t},P_{X^v_{t-\delta}},
v_t,v_{t-\delta},p_t,q_t )-H(t,\Theta_t,p_t,q_t )\nonumber \\
&\geq  &\left<H_x(t,\Theta_t,p_t,q_t),X^{v}_t-{X}^u_t\right>\nonumber \\
&&+\mathbb{E}'\left[\left< H_\mu(t,\Theta_t,({X}^{u}_t)',p_t,q_t,
),(X^{v}_t-{X}^u_t)'\right>\right]\nonumber\\
&&+\left<H_{x_\delta}(t,\Theta_t,p_t,q_t),X^{v}_{t-\delta}-{X}^u_{t-\delta}\right>\nonumber \\
&&+\mathbb{E}'\left[\left< H_{\mu_\delta}(t,\Theta_t,({X}^{u}_t)',p_t,q_t
),(X^{v}_{t-\delta}-{X}^u_{t-\delta})'\right>\right]\nonumber\\
&&+ \left<H_v(t,\Theta_t,p_t,q_t),v_t-u_t\right>+\left<H_{v_\delta}(t,\Theta_t,p_t,q_t),v_{t-\delta}-u_{t-\delta}\right>.\nonumber
\end{eqnarray}

\indent Thus, by maximum condition (\ref{2eq43}), we have
\begin{eqnarray}
&&J(v)-J(u)\nonumber \\
&\geq &\mathbb{E}\int^T_0  \left<H_v(t,\Theta_t,p_t,q_t),v_t-u_t\right>dt\nonumber\\
&&+\mathbb{E}\int^T_0  \left<H_{v_\delta}(t,\Theta_t,p_t,q_t),v_{t-\delta}-u_{t-\delta}\right>dt\nonumber\\
&\geq& 0.
\end{eqnarray}

\indent The above equality complete the proof.

\end{proof}

 \end{document}